\documentclass[leqno]{siamart1116}
\usepackage{amssymb}
\usepackage{graphicx} 
\usepackage{xspace}
\usepackage{multirow		}
\usepackage{Macros/mydef}
\usepackage[numbers]{natbib}
\usepackage{subcaption}
\usepackage{enumitem}
\newsiamthm{example}{Example}
\usepackage{amsmath}
\usepackage{nicematrix}
\usepackage{booktabs}

\let\cite=\citet
\usepackage{accents}

\newcommand{\ttop}{{\mathsf T}}

\newtheorem{remark}{Remark}[section]

\usepackage{xcolor}

\title{
  High-order structure-preserving SBP finite difference methods for the Vlasov--Maxwell system on matrix-free GPUs\thanks{Authors are listed alphabetically} }

\author{
  Robin Dym\'er\thanks{Division of Scientific Computing, Department of Information
    Technology, Uppsala University, Box 524, 751\,05 Uppsala, Sweden
    (\email{robin.dymer@it.uu.se}, \email{ken.mattsson@it.uu.se},
    \email{murtazo.nazarov@it.uu.se}). Corresponding author: Robin Dym\'er.}
  \and Ken Mattsson\footnotemark[2]
  \and Murtazo Nazarov\footnotemark[2]}

\begin{document}

\maketitle

\renewcommand{\thefootnote}{\fnsymbol{footnote}} 
\renewcommand{\thefootnote}{\arabic{footnote}}

\begin{abstract}
  In this paper, we present a high-order, stable summation-by-parts (SBP) finite difference method for solving the Vlasov--Maxwell system in a 2D2V phase space. Central SBP operators for the advection terms are not stable when the solution becomes non-smooth and fine-scale filamentary structures develop, as is typical in high-dimensional Vlasov--Maxwell simulations. To address this issue, the method is stabilized using high-order upwind SBP operators. High-order explicit Runge--Kutta methods are employed for time integration. We prove that the fully discrete scheme exactly conserves mass and preserves momentum up to truncation error. Furthermore, we present a matrix-free implementation of the method on modern GPU architectures. A range of challenging benchmark problems is solved to demonstrate the accuracy, robustness, and performance of the proposed scheme.

\end{abstract}


\pagestyle{myheadings}
\thispagestyle{plain}
\markboth{R. Dym\'er, K. Mattsson, and M. Nazarov}{SBP methods for the Vlasov--Maxwell system}

\section{Introduction}
\label{Sec:introduction}

The simulation of plasma dynamics is an important problem in modern physics and computational mathematics. Understanding and analyzing plasma physics is essential for many applications, including nuclear fusion energy and astrophysics.

Several mathematical models are used to study plasma dynamics. Fluid models, such as magnetohydrodynamics (MHD), introduced by \cite{Alfen_1942}, are among the most well-studied models. These models solve for conservative plasma variables, such as density, momentum, energy, and magnetic field, and are widely used to study large-scale plasma dynamics. However, fluid models such as MHD cannot accurately describe particle interactions at the microscopic level. For instance, important physical effects such as non-equilibrium dynamics, wave-particle interactions, and small-scale turbulence cannot be fully captured by MHD models.

The Vlasov--Maxwell system, introduced by \cite{Vlasov_1968}, provides a fundamental kinetic description of collisionless plasmas by modeling the evolution of particle distribution functions under self-consistent electromagnetic fields. It arises in a wide range of applications, including plasma physics, astrophysics, and controlled fusion, where accurate numerical approximation of multiscale and nonlinear phenomena is essential; see, for example, the book by \cite{Tajima_2004}.

The Vlasov equation is a scalar nonlinear equation for a distribution function in phase space. It is coupled with Maxwell's equations for the electric and magnetic fields. The phase space consists of spatial and velocity variables, resulting in six dimensions in full physical settings. Therefore, although the Vlasov equation is scalar, solving it in phase space is computationally challenging. In addition, fine-scale structures such as filamentation require accurate numerical schemes that preserve key structures of plasma dynamics, such as conservation laws. Failure to capture these structures may lead to numerical instabilities and nonphysical solutions.

Traditional numerical methods for simulating Vlasov equations include particle-in-cell methods, see for instance \cite{Victory_1991, Degond_2010}, and semi-Lagrangian schemes, see \cite{Filbet_2001, Shiroto_et_al_2019} and references therein. Particle-in-cell methods are widely used to simulate plasma dynamics due to their computational efficiency and scalability to higher dimensions. However, it is difficult to make them high-order accurate; they also suffer from statistical noise and may violate important physical structures of the system. Some of these limitations are addressed by semi-Lagrangian methods. However, semi-Lagrangian methods are mostly well suited for smooth problems and may suffer from a lack of conservation.

Grid-based methods have gained increasing attention in the numerical community due to their structure-preserving properties; see, for example, \citep{Jing-Shu_2011, Rossmanith_2011, Hittinger_2013, Kormann_2025} and references therein, mainly for discontinuous Galerkin, finite volume, and finite element approximations.

Finite difference methods are attractive for solving hyperbolic problems due to their computational efficiency and high-order accuracy \citep{LeVeque_2007}. They have been applied to kinetic plasma models, mostly for the Vlasov--Poisson equations; see \citep{Filbet_2003, Banks_et_al_2019} and references therein. However, their application to the fully coupled Vlasov--Maxwell system remains limited. We refer the reader to \citep{Shiroto_et_al_2019}, where a quadratically conservative finite difference scheme is presented for the relativistic Vlasov--Maxwell system.

The method in \citep{Shiroto_et_al_2019} employs summation-by-parts (SBP) operators in space and an implicit second-order method in time. SBP finite difference operators, see, e.g., \citep[Chap.~7]{Gustafsson_2008} and the reviews \citep{Svard_Nordstrom_2014, DelReyFernandez_et_al_2014}, provide a systematic framework for constructing stable high-order discretizations. These operators mimic integration by parts at the discrete level, enabling the derivation of energy estimates and conservation properties analogous to those in the continuous setting.

In this work, we develop a high-order SBP finite difference discretization for the Vlasov--Maxwell system in a two-dimensional phase space setting. The method is constructed using tensor-product SBP operators and incorporates an upwind stabilization mechanism based on a Lax--Friedrichs-type flux splitting, as presented in \citep{Mattsson_2017}. This stabilization introduces controlled numerical dissipation while maintaining the SBP structure.

The contributions of this work are as follows. Through a rigorous analysis, we prove that the proposed method conserves mass, momentum, and the $L^2$ norm of the distribution function in the semi-discrete setting. At the fully discrete level, we analyze time integration using Forward Euler and discuss the extension to explicit Runge--Kutta methods, showing that linear invariants are preserved. We also examine the effect of stabilization on total energy, which is no longer conserved due to the added dissipation.

In addition to the analytical results, we present a matrix-free implementation of the proposed method based on tensor-product stencil evaluations. This approach avoids the assembly of global operators and enables efficient computation in high-dimensional phase space. While the implementation is motivated by computational considerations, the focus of this work is on the underlying numerical analysis and structure-preserving properties of the scheme.

The performance of the method is demonstrated through a series of numerical experiments, including convergence studies and simulations of standard plasma instabilities. The results confirm the expected order of accuracy and illustrate the robustness of the approach.

The remainder of the paper is organized as follows. In Section~\ref{Sec:prelim}, we introduce the Vlasov--Maxwell system and its conservation properties. Section~\ref{Sec:sbp} presents the SBP discretization and establishes the main analytical results. The fully discrete scheme is discussed in Section~\ref{Sec:fully-discrete}. Section~\ref{Sec:gpu} describes the matrix-free implementation. Numerical experiments are presented in Section~\ref{sec:results}, and conclusions are drawn in Section~\ref{Sec:conclusions}.

\section{Preliminaries}
\label{Sec:prelim}
In this section we present the Vlasov--Maxwell system and discuss physical properties of it.

\subsection{Governing equations}
We consider the Vlasov--Maxwell system, which describes the evolution of a collisionless plasma under self-consistent electromagnetic fields. This system governs the evolution of the distribution function $f_s(\bx,\bv ,t)$ for a particle species $s$ in phase space $(\bx, \bv) \in \Omega := \Omega_{\bx} \CROSS \Omega_{\bv}$, where $\Omega_\bx, \Omega_\bv \subseteq \mathbb{R}^d$, and $d$ is the spatial dimension. The system satisfies the following equation:
\begin{equation}\label{eq:vm}
  \begin{aligned}
    \p_t f_s+\bv \SCAL\GRAD_{\bx }f_s +\frac{q}{m}(\bE +\bv \CROSS\bB )\SCAL\GRAD_{\bv }f_s  &\, = 0, 
                                                                                               \quad (\bx ,\bv ,t)\in\Omega\CROSS \polR^+, \\
    f_s (\bx ,\bv ,0) &\, =  f_{s,0}(\bx ,\bv ),
                        \quad (\bx ,\bv )\in\Omega,
  \end{aligned}
\end{equation}
where $f_{s,0}(\bx ,\bv )$ is a given initial condition, and $q_s$ and $m_s$ represent the charge and mass of species $s$. We use periodic boundary conditions in both the $\bx$ and $\bv$ directions when the phase space is bounded. 

The electric field $\bE (\bx ,t)$ and magnetic field $\bB (\bx ,t)$ are governed by Maxwell’s equations:
\begin{equation}\label{eq:maxwell}
  \begin{aligned}
    \frac{1}{c^2}\p_t \bE  &= \GRAD_{\bx} \CROSS \bB  - \mu_0 \bJ , \\
    \p_t \bB  &= -\GRAD_{\bx} \CROSS \bE, \\
    \bE(\bx, 0) &= \bE_0(\bx), \\
    \bB(\bx, 0) &= \bB_0(\bx),
  \end{aligned}
\end{equation}
where $\bE_0(\bx)$ and $\bB_0(\bx)$ are given initial conditions, $c = (\e_0 \mu_0)^{-1/2}$ is the speed of light in vacuum, and $\e_0$ and $\mu_0$ denote the permittivity and permeability, respectively. The charge density $\rho(\bx ,t)$ and current density $\bJ (\bx ,t)$ are defined by
\begin{equation}\label{eq:density}
  \bJ (\bx ,t) = \sum_s q_s \int_{\polR^{d}} \bv  f_s(\bx ,\bv ,t) \ud \bv ,
  \qquad
  \rho(\bx ,t) = \sum_s q_s \int_{\polR^{d}} f_s(\bx ,\bv ,t) \ud \bv .
\end{equation}

The electric and magnetic fields should also satisfy the following Gauss' laws:
\begin{equation}\label{eq:gauss}
  \begin{aligned}
    \GRAD_{\bx } \SCAL \bE  &= \frac{\rho}{\e_0}, \\
    \GRAD_{\bx } \SCAL \bB  &= 0.
  \end{aligned}
\end{equation}

In this work, we adopt nondimensional units and set these constants to unity. We also focus on a single species with $s=1$ and assume $q_s/m_s \equiv 1$. We further drop the subscript and denote the distribution function by $f(\bx, \bv, t)$.

\subsection{Vlasov--Maxwell equations in 2D2V}
\label{Sec:VM-2D2V}

For the rest of the paper, we set $d=2$ and study the Vlasov--Maxwell system \eqref{eq:vm}--\eqref{eq:maxwell} in two spatial and two velocity (2D2V) dimensions. First, we show how the equations can be written in this case, and then we derive stability estimates in the semi-discrete setting. 

Let the domain $\Omega$ be defined as
\[
  \Omega_{\bx} := [0, L_x] \CROSS [0, L_y], \quad \Omega_{\bv}:= [0, L_v] \CROSS [0, L_w],
\]
where we define $\bx := (x,y)$ and $\bv := (v,w)$ as the spatial and velocity coordinates, respectively. We denote the electromagnetic fields by $\bE = (E_x, E_y, E_z)$ and $\bJ = (J_x, J_y, J_z)$, and let $\bB = (B_x, B_y, B_z)$. Under the 2D2V reduction, only the out-of-
plane magnetic field component $B_z$ is retained. Then, the reduced 2D2V Vlasov--Maxwell system for the unknown vector $\bU := (f, E_x, B_z, E_y)^\ttop$ can be written as
\begin{equation}\label{eq:2d2v}
  \begin{cases}
    \begin{aligned}
      \p_t f =&\ -v\p_x f - w\p_y f - \big(E_x + w B_z \big) \p_v f - \big(E_y - vB_z \big) \p_w f ,\\
      \p_t E_x =&\ \p_y B_z - J_x,\\
      \p_t B_z =&\ \p_y E_x - \p_x E_y,\\
      \p_t E_y =&\ -\p_x B_z - J_y.
    \end{aligned}
  \end{cases}
\end{equation}

It is often useful to write the above system in matrix form:
\begin{equation}\label{eq:system2D2V}
  \begin{aligned}
    \p_t \bU = \sfA_x \p_x \bU + \sfA_y \p_y \bU + \sfA_{v} \p_v \bU + \sfA_{w} \p_w \bU + F&, \quad (\bx, \bv, t) \in \Omega\CROSS \polR^+, \\
    \bU(0) = \bU_0&,\quad (\bx, \bv) \in \Omega,
  \end{aligned}
\end{equation}
where the coefficient matrices are defined as 
\begin{equation*}
  F =
  \begin{pmatrix}
    0\\
    -J_x\\
    0\\
    -J_y
  \end{pmatrix},
  \quad
  \sfA_x =
  \begin{pmatrix}
    -v & 0 & 0 & 0 \\
    0  & 0 & 0 & 0 \\
    0  & 0 & 0 & -1 \\
    0  & 0 & -1 & 0
  \end{pmatrix},
  \quad
  \sfA_y =
  \begin{pmatrix}
    -w & 0 & 0 & 0 \\
    0  & 0 & 1 & 0 \\
    0  & 1 & 0 & 0 \\
    0  & 0 & 0 & 0
  \end{pmatrix},
\end{equation*}

\begin{equation*}
  \sfA_v =
  \begin{pmatrix}
    -\big(E_x + wB_z\big) & 0 & 0 & 0 \\
    0 & 0 & 0 & 0 \\
    0 & 0 & 0 & 0 \\
    0 & 0 & 0 & 0
  \end{pmatrix},
  \quad
  \sfA_w =
  \begin{pmatrix}
    -\big(E_y - vB_z\big) & 0 & 0 & 0 \\
    0 & 0 & 0 & 0 \\
    0 & 0 & 0 & 0 \\
    0 & 0 & 0 & 0
  \end{pmatrix}.
\end{equation*}
Here $F = F(\bU)$ contains the current coupling through Maxwell's equations. In 2D2V, the current density in \eqref{eq:density} reduces to
\[
  J_{x} = \int_{\Omega_{\bv}} vf \ud \bv,\quad
  J_{y} = \int_{\Omega_{\bv}} wf \ud \bv.
\]

\begin{proposition}
  The Vlasov--Maxwell system preserves the following conservation properties:
  \begin{enumerate}
  \item Conservation of mass:
    \[
      \p_t \int_\Omega f \ud \bx \ud \bv = 0.
    \]

  \item Conservation of momentum:
    \[
      \p_t \Big(\int_\Omega f \bv \ud \bx \ud \bv
      + \int_{\Omega_\bx} \bE \CROSS \bB \ud \bx \Big) = 0.
    \]
    
  \item Conservation of total energy:
    \[
      \p_t
      \left(
        \frac12 \int_{\Omega} f |\bv|^2 \ud \bx \ud \bv
        +
        \frac12 \int_{\Omega_{\bx}} \big(|\bE|^2 + |\bB|^2\big)\ud\bx
      \right)
      = 0.
    \]

  \item Conservation of the $L^2$-norm of $f$:
    \[
      \p_t \|f\|^2 = 0.
    \]

  \end{enumerate}
\end{proposition}

\begin{proof}
  The conservation of mass is obtained by integrating the Vlasov equation. The conservation of momentum is a bit involved. Since we later establish it for our numerical scheme, we provide the complete derivation here. The conservation of momentum for the reduced 2D2V system \eqref{eq:2d2v} can be written as
  \begin{equation}\label{eq:cons:mom}
    \p_t
    \left(
      \int_{\Omega}
      \begin{pmatrix}
        v\\ w
      \end{pmatrix}
      f\,\ud\bv \ud\bx
      +
      \int_{\Omega_x}
      \begin{pmatrix}
        E_yB_z\\
        -E_xB_z
      \end{pmatrix}
      \ud\bx
    \right)
    =0.
  \end{equation}
  Here, the first term is usually referred to as {\em the particle momentum}, and the second term is referred to as {\em the electromagnetic momentum}.
  
  By multiplying the Vlasov equation by $\bv$ and integrating, we get:
  \[
    \begin{aligned}
      \int_\Omega  
      \begin{pmatrix}
        v\\ w
      \end{pmatrix}
      & \p_t f
        \ud \bx \ud \bv \\ 
      = 
      & \int_\Omega 
        \begin{pmatrix}
          v\\ w
        \end{pmatrix}
        \Big(
        -v\p_x f - w\p_y f - \big(E_x + w B_z \big) \p_v f - \big(E_y - vB_z \big) \p_w f
        \Big) \ud \bx \ud \bv.
    \end{aligned}
  \]
  
  The spatial derivative terms vanish after integration by parts under periodic boundary conditions. The velocity terms containing  $\p_w v $ and $\p_v w$ vanish after integration by parts. For the remaining terms, we write
  \[
    - \int_\Omega v \big(E_x + w B_z \big) \p_v f \ud \bx \ud \bv 
    =
    \int_\Omega \big(E_x + w B_z \big) f \ud \bx \ud \bv 
    =
    \int_{\Omega{_\bx}} \big(E_x \rho + J_y B_z \big) \ud \bx,
  \]
  and
  \[
    - \int_\Omega w \big(E_y - v B_z \big) \p_w f \ud \bx \ud \bv 
    =
    \int_\Omega \big(E_y - v B_z \big) f \ud \bx \ud \bv 
    =
    \int_{\Omega{_\bx}} \big(E_y \rho - J_x B_z \big) \ud \bx.
  \]
  
  Thus,
  \[
    \begin{aligned}
      \p_t
      \Big(
      \int_\Omega  
      \begin{pmatrix}
        v\\ w
      \end{pmatrix}
      f
      \ud \bx \ud \bv
      \Big)  
      = 
      \begin{pmatrix}
        \int_{\Omega{_\bx}} \big(E_x \rho + J_y B_z \big) \ud \bx \\
        \int_{\Omega{_\bx}} \big(E_y \rho - J_x B_z \big) \ud \bx.
      \end{pmatrix}
    \end{aligned}
  \]

  Next, using the Maxwell equations \eqref{eq:2d2v} and Gauss' laws \eqref{eq:gauss}, we can verify that 
  \begin{equation}\label{eq:elec:mom}
    \p_t
    \begin{pmatrix}
      E_y B_z \\
      -E_x B_z
    \end{pmatrix}
    =
    -
    \begin{pmatrix}
      E_x \rho + J_y B_z \\
      E_y \rho - J_x B_z,
    \end{pmatrix}
  \end{equation}
  which completes the proof of the momentum conservation \eqref{eq:cons:mom}.

  We next establish conservation of the total energy and the $L^2$-norm of $f$. Multiplying the Maxwell equations by the corresponding field components and integrating over the spatial domain, we obtain
  \[
    \frac12 \p_t \int_{\Omega_{\bx}} \big(|\bE|^2 + |B_z|^2\big)\ud\bx
    =
    - \int_{\Omega_{\bx}} \bJ \SCAL \bE \ud \bx,
  \]
  where periodic boundary conditions have been used to eliminate boundary terms.

  On the other hand, multiplying the Vlasov equation in \eqref{eq:2d2v} by
  $\frac12|\bv|^2$ and integrating over phase space gives the kinetic energy balance
  \[
    \frac12 \p_t \int_{\Omega} f |\bv|^2 \ud \bx \ud \bv
    =
    \int_{\Omega_{\bx}} \bJ \SCAL \bE \ud \bx .
  \]
  Adding the two identities yields conservation of total energy.

  Moreover, multiplying the Vlasov equation in \eqref{eq:2d2v} by $f$ and integrating over phase space yields conservation of the $L^2$-norm of $f$:
  \[
    \p_t \|f\|^2 = 0.
  \]  

\end{proof}

In the above analysis, we repeatedly use integration by parts together with periodic boundary conditions. The SBP operators used in this work mimic integration by parts at the discrete level; therefore, the above conservation properties are expected to carry over to the discrete level. One objective of this work is to design a numerical discretization that preserves these properties.

\section{Summation-by-parts discretization}
\label{Sec:sbp}

Since higher dimensional operators are constructed by taking a Kronecker product of one-dimensional operators, we give the definition of SBP operators in 1D. Let us consider a domain $D:= [0, L]$ discretized by $m$ equal-length intervals of size $h$:
\[
  x_i = i h, \quad i = 0,1,\dots,m, \quad h = \frac{L}{m}.
\]
Let $\ve_l$ and $\ve_r$ denote the following vectors in $\polR^{m+1}$:
\begin{equation}\label{eq:e_lr}
  \ve_l = \big(1, 0, \ldots, 0 \big)^\ttop,
  \quad
  \ve_r = \big(0, 0, \ldots, 1 \big)^\ttop.
\end{equation}
We define 
\begin{equation}
  \sfB = - \frac{1}{2} \ve_l \ve_l^\ttop + \frac{1}{2} \ve_r \ve_r^\ttop,
  \label{eq:B_definition}
\end{equation}
which will be useful below. 

\medskip
\begin{definition}
  A diagonal quadrature matrix $\sfH$ is said to define a discrete norm if it is symmetric positive definite and satisfies
  $
  \mathbf{1}^\top \sfH \mathbf{1} = L,
  $
  where $\mathbf{1}$ denotes the vector of ones in $\polR^{m+1}$.
\end{definition}

For given functions $f\in L^2(0,L)$ and $g\in L^2(0,L)$ we construct vectors 
$\vf:=\big( f(x_0), \ldots, f(x_m) \big)^\ttop$ and $\vg:=\big( g(x_0), \ldots, g(x_m) \big)^\ttop$. Then, the $L^2$-inner product is discretely approximated by
\[
  (f,g) \approx (\vf, \vg)_{\sfH},
\]
where
\[
  (\vf, \vg)_{\sfH} := \vf^\ttop \sfH \vg = 
  \sum_{i=0}^m \sfH_{ii} f(x_i) g(x_i).
\]

With this definition, we can approximate the integral of a function by 
\[
  \int_\Omega g(x) \ud x \approx (\mathbf{1}, \vg)_{\sfH} = 
  \mathbf{1}^\ttop \sfH \vg = \sum_{i=0}^m \sfH_{ii} g(x_i).
\]
We define the discrete $L^2$-norm by
\[
  \|\vg\|_{\sfH}^2 := (\vg, \vg)_{\sfH}.
\]

\medskip
\begin{definition}\label{def:sbp} 
  We say that a difference operator $\sfD_1$ approximating $\p_x$ is a $p$th-order accurate interior SBP operator with quadrature matrix $\sfH$ if $\sfH = \sfH^\ttop > 0$, and there exists a skew-symmetric matrix $\sfQ$, \ie $\sfQ + \sfQ^\ttop = 0$, such that
  \[
    \sfD_1 = \sfH^{-1} \big(\sfQ + \sfB \big),
  \]
  where $\sfB$ is defined in \eqref{eq:B_definition}.
\end{definition}

The SBP operator $\sfD_1$ is a nondissipative central-difference approximation of the first-order derivative. Therefore, when applying it to advection-dominated problems such as the Vlasov problem, high-frequency modes are not damped. A remedy for this problem is to use artificial dissipation, which ideally damps high-frequency modes; see, \eg \citep{Stiernstrom_et_al_2021}, where residual-based artificial viscosity is combined with upwind SBP operators for scalar conservation laws. To suppress these spurious modes, we employ upwind SBP operators, which add controlled numerical dissipation while preserving the SBP stability framework \citep{Mattsson_2017}.

\begin{definition}\label{def:sbp_upwind}
  We say that the difference operators $\sfD_\pm$ approximating $\p_x$ are first-derivative upwind SBP operators with quadrature matrix $\sfH$ if $\sfH = \sfH^\ttop > 0$, $\sfQ_+ + \sfQ_-^\ttop = 0$, $\sfQ_+ + \sfQ_+^\ttop = 2\sfS \leq 0$ (negative semidefinite), and
  \[
    \sfD_\pm = \sfH^{-1} \big(\sfQ_\pm + \sfB \big).
  \]
\end{definition}

We have the following relations between upwind and central SBP operators:
\begin{equation}\label{eq:upwind_identity}
  \begin{aligned}
    \frac{\sfD_++\sfD_-}{2} &= \sfH^{-1}\Big( \frac{\sfQ_+ + \sfQ_-}{2} + \sfB \Big) = \sfH^{-1}\Big( \sfQ + \sfB \Big) = \sfD_1,\\
    \frac{\sfD_+-\sfD_-}{2} &= \sfH^{-1}\Big( \frac{\sfQ_+ - \sfQ_-}{2} \Big) = \sfH^{-1}\Big( \frac{\sfQ_+ + \sfQ_+^\ttop}{2} \Big) = \sfH^{-1}\sfS.
  \end{aligned}
\end{equation}

In higher dimensions, the 1D finite difference operators need to be extended using the Kronecker product. For example, in two space dimensions, if $\sfA$ is an $m\CROSS n$ matrix and $\sfC$ is a $p\CROSS q$ matrix, then the Kronecker product $\sfA \otimes \sfC$ is an $mp \CROSS nq$ matrix of the form
\begin{equation}\label{eq:kronecker}
  \sfA \OTIMES \sfC =
  \begin{bmatrix}
    \sfA_{11} \sfC & \sfA_{12} \sfC & \cdots & \sfA_{1n} \sfC \\
    \sfA_{21} \sfC & \sfA_{22} \sfC & \cdots & \sfA_{2n} \sfC \\
    \vdots   & \vdots   & \ddots & \vdots   \\
    \sfA_{m1} \sfC & \sfA_{m2} \sfC & \cdots & \sfA_{mn} \sfC
  \end{bmatrix}.
\end{equation}

Using the above extension, we can easily discretize the phase space $\Omega$ in 2D2V. Let us assume $\Omega_\bx$ is a rectangular domain $[0, L_x]\times[0, L_y]$, and $\Omega_\bv$ is a rectangular domain $[0, L_v]\times[0, L_w]$. The domain $\Omega_\bx$ is discretized with an $(m_x + 1) \times (m_y + 1)$-point grid, and the domain $\Omega_\bv$ is discretized with an $(m_v + 1) \times (m_w + 1)$-point grid, with grid points defined as 
\begin{align*}
  x_i &= i h_x, \quad i = 0,1,\dots,m_x, \quad h_x = L_x/m_x, \\ 
  y_j &= j h_y, \quad j = 0,1,\dots,m_y, \quad h_y = L_y/m_y, \\
  v_k &= k h_v, \quad k = 0,1,\dots,m_v, \quad h_v = L_v/m_v, \\ 
  w_l &= l h_w, \quad l = 0,1,\dots,m_w, \quad h_w = L_w/m_w. \\
\end{align*}

Let $\sfI_r$ be the identity matrix of size $r$. Then, the difference operator $\sfD_1$ can be extended to 2D2V using the following directional derivatives:
\begin{equation}\label{eq:extension}
  \begin{aligned}
    \sfD_x &= \sfI_4 \otimes \sfD_1 \otimes \sfI_{m_y} \otimes \sfI_{m_v} \otimes \sfI_{m_w},\\
    \sfD_y &= \sfI_4 \otimes \sfI_{m_x} \otimes \sfD_1 \otimes \sfI_{m_v} \otimes \sfI_{m_w},\\
    \sfD_v &= \sfI_4 \otimes \sfI_{m_x} \otimes \sfI_{m_y} \otimes \sfD_1 \otimes \sfI_{m_w},\\
    \sfD_w &= \sfI_4 \otimes \sfI_{m_x} \otimes \sfI_{m_y} \otimes \sfI_{m_v} \otimes \sfD_1.
  \end{aligned}
\end{equation}
Similarly, the upwind SBP operators $\sfD_\pm$ are extended to 2D2V as
\begin{align*}
  \sfD_{x_\pm} &= \sfI_4 \otimes \sfD_{\pm} \otimes \sfI_{m_y} \otimes \sfI_{m_v} \otimes \sfI_{m_w},\\
  \sfD_{y_\pm} &= \sfI_4 \otimes \sfI_{m_x} \otimes \sfD_{\pm} \otimes \sfI_{m_v} \otimes \sfI_{m_w},\\
  \sfD_{v_\pm} &= \sfI_4 \otimes \sfI_{m_x} \otimes \sfI_{m_y} \otimes \sfD_{\pm} \otimes \sfI_{m_w},\\
  \sfD_{w_\pm} &= \sfI_4 \otimes \sfI_{m_x} \otimes \sfI_{m_y} \otimes \sfI_{m_v} \otimes \sfD_{\pm}.
\end{align*}

\subsection{Semi-discrete approximation of Vlasov--Maxwell}

Now, we are ready to discretize the Vlasov--Maxwell system in phase space. Let $\calN$ be the set of all grid multi-indices $\bn = (i,j,k,l)$, ordered lexicographically (i.e., in row-major order). We define
\[
  \begin{aligned}
    \vU(t)
    &:=
      (\bU_{\bn}(t))^\ttop_{\bn \in \calN},
      \mbox{ where }
      \bU_{\bn}(t)
      =
      \big(
      f_{\bn}(t),
      E_{x,\bn}(t),
      B_{z,\bn}(t),
      E_{y,\bn}(t)
      \big)^\ttop,
  \end{aligned}
\]
to be the grid-function approximation of $\bU(\bx, \bv,t)$ at time $t$. Then, the semi-discrete approximation of \eqref{eq:system2D2V} is given by
\begin{equation}\label{eq:sbp_2d2v}
  \begin{aligned}
    \p_t\vU &= 
              \big( 
              \overline{\sfA}_x \sfD_x +
              \overline{\sfA}_y \sfD_y +
              \overline{\sfA}_v \sfD_v +
              \overline{\sfA}_w \sfD_w
              \big)\vU + \overline{\sfF}, \quad t\in \polR^+ \\
    \vU(0) &= \vU_{0},
  \end{aligned}
\end{equation} 
where $\vU_{0}$ is the value of the initial condition at the grid points, and the coefficient matrices are defined as 
\[
  \begin{aligned}
    \overline{\sfA}_x &= \sfI_{m_x} \otimes \sfI_{m_y} \otimes \operatorname{diag}(-\bv) \otimes \sfI_{m_{w}} \otimes \sfone_f + \sfI_N \otimes \sfA^C_x, \\
    \overline{\sfA}_y &= \sfI_{m_x} \otimes \sfI_{m_y} \otimes \sfI_{m_v} \otimes \operatorname{diag}(-\bw) \otimes \sfone_f + \sfI_N \otimes \sfA^C_y, \\
    \overline{\sfA}_v &= \left[ \operatorname{diag}(-\bE_x) \otimes \sfI_{m_{v}} \otimes \sfI_{m_{w}} + \operatorname{diag}(-\bB_z) \otimes \sfI_{m_v} \otimes \operatorname{diag}(\bw)\right] \otimes \sfone_f, \\
    \overline{\sfA}_w &= \left[ \operatorname{diag}(-\bE_y) \otimes \sfI_{m_{v}} \otimes \sfI_{m_{w}} + \operatorname{diag}(\bB_z) \otimes \operatorname{diag}(\bv) \otimes \sfI_{m_w}\right] \otimes \sfone_f, \\
    \overline{\sfF}   &= \sfH_\Omega^{-1} \vF,
  \end{aligned}\]
where $\sfI_N = \sfI_{m_x} \otimes \sfI_{m_y} \otimes \sfI_{m_v} \otimes \sfI_{m_w}$ and $\sfone_f = \operatorname{diag}(1, 0, 0, 0)$.
The matrix $\sfone_f$ is used to pad the matrix to the correct dimensions. Furthermore, $\sfA^C_x$ and $\sfA^C_y$ are the
$\sfA_x$ and $\sfA_y$ matrices without the variable parts. What these operations simply do is discretize each entry in the continuous matrices
in phase space, yielding matrices with entries that are defined on the phase space grid. Lastly,
\[
  \sfH_\Omega
  =
  \sfI_4 \otimes \sfH_{m_x} \otimes \sfH_{m_y} \otimes \sfH_{m_{v}} \otimes \sfH_{m_{w}},
\]
and
\[
  \begin{aligned}
    \vF(t)
    &:=
      (\bF_{\bn}(t))^\ttop_{\bn \in \calN},
      \mbox{ where }
      \bF_{\bn}(t)
      =
      \big(
      0,
      -J_{x,\bn}(t),
      0,
      -J_{y,\bn}
      \big)^\ttop.
  \end{aligned}
\]

As mentioned earlier, the SBP operators in \eqref{eq:sbp_2d2v} are central difference schemes; therefore, we need to stabilize them. Upwind SBP operators can be used to stabilize the semi-discrete approximation \eqref{eq:sbp_2d2v}. However, the coefficient matrices $\overline{\sfA}_z$, $z=\{x,y,v,w\}$, need to be split into a part with non-negative eigenvalues, \eg $\overline{\sfA}_{z,+}$, and a part with non-positive eigenvalues, \eg $\overline{\sfA}_{z,-}$, such that $\overline{\sfA}_{z} = \overline{\sfA}_{z,+} + \overline{\sfA}_{z,-}$. We follow \citep{Mattsson_2017} and apply the Lax--Friedrichs splitting in the form
\[
  \overline{\sfA}_{z, \pm} = \frac12 \big( \overline{\sfA}_{z} \pm k_s \alpha_{\overline{\sfA}_{z}} \sfI \big),
\]
where $\alpha_{\overline{\sfA}_z}$ is the maximum absolute eigenvalue of $\overline{\sfA}_z$, and it can be multiplied by a parameter $k_s\ge1$ to increase numerical dissipation, and $\sfI$ is the identity matrix of compatible size for $\overline{\sfA}_z$. Now, we apply this flux-splitting algorithm to the advection terms of \eqref{eq:sbp_2d2v} and obtain
\[
  \begin{aligned}
    (\overline{\sfA}_{z,+} \sfD_{z,+} + \overline{\sfA}_{z,-} \sfD_{z,-})\vU 
    &= 
      \frac{1}{2}\Big( (\overline{\sfA}_z + k_s \alpha_{\overline{\sfA}_z} \sfI)\sfD_{z,+} + (\overline{\sfA}_z - k_s\alpha_{\overline{\sfA}_z} \sfI)\sfD_{z,-} \Big)\vU \\
    &= 
      \frac{1}{2} \overline{\sfA}_z \big(\sfD_{z,+} + \sfD_{z,-}\big)\vU + \frac{1}{2} k_s\alpha_{\overline{\sfA}_z} \big(\sfD_{z,+} - \sfD_{z,-}\big)\vU \\
    &= 
      \overline{\sfA}_z \sfD_{z} \vU + k_s\alpha_{\overline{\sfA}_z} \sfH_z^{-1}\sfS_z \vU.
  \end{aligned}
\]
Here the matrices $\sfH_z$ and $\sfS_z$ are extended to higher dimensions in the same way as the matrix $\sfD_z$ in \eqref{eq:extension} for $z=\{x,y,v,w\}$.

Collecting the terms after splitting, we obtain a stabilized SBP formulation of the Vlasov--Maxwell equations:
\begin{equation}\label{eq:upwind:sbp_2d2v}
  \begin{aligned}
    \p_t\vU &= 
              \big( 
              \overline{\sfA}_x \sfD_x +
              \overline{\sfA}_y \sfD_y +
              \overline{\sfA}_v \sfD_v +
              \overline{\sfA}_w \sfD_w
              \big)\vU 
    \\
            &+
              k_s\big( 
              \alpha_{\overline{\sfA}_x} \sfH_x^{-1}\sfS_x +
              \alpha_{\overline{\sfA}_y} \sfH_y^{-1}\sfS_y +
              \alpha_{\overline{\sfA}_v} \sfH_v^{-1}\sfS_v +
              \alpha_{\overline{\sfA}_w} \sfH_w^{-1}\sfS_w
              \big)\vU + \overline{\sfF}, \quad t > 0, \\
    \vU(0) &= \vU_{0}.
  \end{aligned}
\end{equation}

It is convenient to combine central difference and stabilization terms in two operators $\calA$ and $\calD$, and write the semi-discrete formulation as
\begin{equation}\label{eq:upwind:sbp_2d2v:short}
  \begin{aligned}
    \p_t\vU &= \calA\vU + \calD\vU + \overline{\sfF}, \quad t >0, \\
    \vU(0) &= \vU_{0}.
  \end{aligned}
\end{equation}

Discrete charge and current densities are computed as
\begin{equation}\label{eq:disc:rho:J}
  \brho := (\vone_{m_v \cdot m_w}, \vf)_{\sfH_{m_v}\otimes\sfH_{m_w}}
  \mbox{ and }
  \vJ := (\vv, \vf)_{\sfH_{m_v}\otimes\sfH_{m_w}}.
\end{equation}

\begin{theorem}\label{thm:semi:cons}
  The semi-discrete upwind SBP approximation of the Vlasov--Maxwell equation \eqref{eq:upwind:sbp_2d2v} conserves the mass and momentum. 
\end{theorem}

\begin{proof} We prove each of the properties separately. 

  \noindent {\bf Mass conservation}. 
  Let us define the vector
  \[
    \mathbf {1}_f
    =
    \mathbf {1} \otimes (1, 0, 0, 0)^\ttop.
  \]
  Now, multiplying the vector $\mathbf{1}_f$ to \eqref{eq:upwind:sbp_2d2v} with respect to the discrete inner product on $\sfH_\Omega$, we obtain:
  \[
    \begin{aligned}
      (\mathbf{1}_f, \p_t\vU)_{\sfH_\Omega} &= 
                                              \big(\mathbf{1}_f, 
                                              \big( 
                                              \overline{\sfA}_x \sfD_x +
                                              \overline{\sfA}_y \sfD_y +
                                              \overline{\sfA}_v \sfD_v +
                                              \overline{\sfA}_w \sfD_w
                                              \big)\vU
                                              \big)_{\sfH_\Omega} 
      \\
                                            &+
                                              \big(\mathbf{1}_f, 
                                              k_s\big( 
                                              \alpha_{\overline{\sfA}_x} \sfH_x^{-1}\sfS_x +
                                              \alpha_{\overline{\sfA}_y} \sfH_y^{-1}\sfS_y +
                                              \alpha_{\overline{\sfA}_v} \sfH_v^{-1}\sfS_v +
                                              \alpha_{\overline{\sfA}_w} \sfH_w^{-1}\sfS_w
                                              \big)\vU + \overline{\sfF}
                                              \big)_{\sfH_\Omega}.
    \end{aligned}
  \]

  Observe that $\big(\mathbf{1}_f, \overline{\sfF} \big)_{\sfH_\Omega} = 0$. Then, for each $z\in \{x,y,v,w\}$, we get
  \[
    \begin{aligned}
      \big(\mathbf{1}_f, \overline{\sfA}_z \sfD_z \vU \big)_{\sfH_\Omega} 
      &+
        \big(\mathbf{1}_f, k_s \alpha_{\overline{\sfA}_z} \sfH_z^{-1}\sfS_z \vU   \big)_{\sfH_\Omega} \\
      &=
        -\big(\sfD_z \mathbf{1}_f, \overline{\sfA}_z \vU \big)_{\sfH_\Omega}
        -
        \frac12 k_s \alpha_{\overline{\sfA}_z}
        \big((\sfD_{z,+} - \sfD_{z,-})\mathbf{1}_f, \vU)
        = 0. 
    \end{aligned}
  \]

  Thus,
  \[
    \begin{aligned}
      \p_t(\vone, \vf)_{\sfH_\Omega} = 0.
    \end{aligned}
  \]

  \noindent {\bf Momentum conservation}.  
  Let us define the vectors
  \[
    \begin{aligned}
      \vv_f
      &:=
        (\vv_{\bn})^\ttop_{\bn \in \calN},
        \mbox{ where }
        \vv_{\bn}
        =
        \big(
        v_{\bn},
        0,
        0,
        0
        \big)^\ttop,
    \end{aligned}
  \]
  and
  \[
    \begin{aligned}
      \vw_f
      &:=
        (\vw_{\bn})^\ttop_{\bn \in \calN},
        \mbox{ where }
        \vw_{\bn}
        =
        \big(
        w_{\bn},
        0,
        0,
        0
        \big)^\ttop,
    \end{aligned}
  \]
  We also define the vectors of velocity coordinates:
  \[
    \vv = (v_{\bn})^\ttop_{\bn \in \calN} \mbox{ and }   
    \vw = (w_{\bn})^\ttop_{\bn \in \calN}. 
  \]

  We multiply $\vv_f$ and $\vw_f$ by \eqref{eq:upwind:sbp_2d2v} with respect to the discrete inner product on $\sfH_\Omega$, and using the definition of the discrete charge and current densities \eqref{eq:disc:rho:J}, we obtain:
  \[
    \begin{aligned}
      (\vv_f, \p_t\vU)_{\sfH_\Omega} &= 
                                       \big(\vv_f, 
                                       \big( 
                                       \overline{\sfA}_x \sfD_x +
                                       \overline{\sfA}_y \sfD_y +
                                       \overline{\sfA}_v \sfD_v +
                                       \overline{\sfA}_w \sfD_w
                                       \big)\vU
                                       \big)_{\sfH_\Omega} 
      \\
                                     &+
                                       \big(\vv_f, 
                                       k_s\big( 
                                       \alpha_{\overline{\sfA}_x} \sfH_x^{-1}\sfS_x +
                                       \alpha_{\overline{\sfA}_y} \sfH_y^{-1}\sfS_y +
                                       \alpha_{\overline{\sfA}_v} \sfH_v^{-1}\sfS_v +
                                       \alpha_{\overline{\sfA}_w} \sfH_w^{-1}\sfS_w
                                       \big)\vU + \overline{\sfF}
                                       \big)_{\sfH_\Omega}, \\
      (\vw_f, \p_t\vU)_{\sfH_\Omega} &= 
                                       \big(\vw_f, 
                                       \big( 
                                       \overline{\sfA}_x \sfD_x +
                                       \overline{\sfA}_y \sfD_y +
                                       \overline{\sfA}_v \sfD_v +
                                       \overline{\sfA}_w \sfD_w
                                       \big)\vU
                                       \big)_{\sfH_\Omega} 
      \\
                                     &+
                                       \big(\vw_f, 
                                       k_s\big( 
                                       \alpha_{\overline{\sfA}_x} \sfH_x^{-1}\sfS_x +
                                       \alpha_{\overline{\sfA}_y} \sfH_y^{-1}\sfS_y +
                                       \alpha_{\overline{\sfA}_v} \sfH_v^{-1}\sfS_v +
                                       \alpha_{\overline{\sfA}_w} \sfH_w^{-1}\sfS_w
                                       \big)\vU + \overline{\sfF}
                                       \big)_{\sfH_\Omega}.
    \end{aligned}
  \]

  Let us discuss each term of the first equality separately.   

  \medskip
  \noindent {\em Upwind terms:}
  Observe that 
  $\big(\vv_f, \overline{\sfF} \big)_{\sfH_\Omega} = 0$ and 
  $\big(\vw_f, \overline{\sfF} \big)_{\sfH_\Omega} = 0$, since the nonzero entries of $\vv_f$ and $\bw_f$ lie only in the distribution-function component,
  whereas $\overline{\sfF}$ has support only in the Maxwell components.

  Further, for each $z=\{x,y,v,w\}$, we get
  \[
    \begin{aligned}
      \big(\vv_f, k_s \alpha_{\overline{\sfA}_z} \sfH_z^{-1}\sfS_z \vU   \big)_{\sfH_\Omega} 
      =
      \frac12 k_s \alpha_{\overline{\sfA}_z}
      \big((\sfD_{z,+} - \sfD_{z,-})\vv_f, \vU)
      = 0.
    \end{aligned}
  \]

  \medskip
  \noindent {\em The time-derivative term:}
  \[
    \begin{aligned}
      (\vv_f, \p_t\vU)_{\sfH_\Omega} &= \p_t  (\vv, \vf)_{\sfH_\Omega}, \\
      (\vw_f, \p_t\vU)_{\sfH_\Omega} &= \p_t  (\vw, \vf)_{\sfH_\Omega}.
    \end{aligned}
  \]

  \medskip
  \noindent {\em The advection term on $\bx$:}
  \[
    \begin{aligned}
      \big(\vv_f, 
      \big( 
      \overline{\sfA}_x \sfD_x +
      \overline{\sfA}_y \sfD_y 
      \big)\vU
      \big)_{\sfH_\Omega} = 0,\\
      \big(\vw_f, 
      \big( 
      \overline{\sfA}_x \sfD_x +
      \overline{\sfA}_y \sfD_y
      \big)\vU
      \big)_{\sfH_\Omega} = 0,
    \end{aligned}
  \]
  due to the SBP property and periodic boundary conditions. 

  \medskip
  \noindent {\em The advection term on $\bv$:}
  Thanks to the SBP property and periodic boundary conditions we get:
  \[
    \begin{aligned}
      \big(\vv_f, 
      \big( 
      \overline{\sfA}_v \sfD_v +
      \overline{\sfA}_w \sfD_w 
      \big)\vU
      \big)_{\sfH_\Omega} 
      &=
        \big(\vv_f, 
        \overline{\sfA}_v \sfD_v
        \vU
        \big)_{\sfH_\Omega}, \\ 
      \big(\vw_f, 
      \big( 
      \overline{\sfA}_v \sfD_v +
      \overline{\sfA}_w \sfD_w 
      \big)\vU
      \big)_{\sfH_\Omega} 
      &=
        \big(\vw_f, 
        \overline{\sfA}_w \sfD_w
        \vU
        \big)_{\sfH_\Omega}.  
    \end{aligned}
  \]
  The first integral can be written as
  \[
    \big(\vv_f, 
    \overline{\sfA}_v \sfD_v
    \vU
    \big)_{\sfH_\Omega}
    =
    \big(\vv, 
    \overline{\sfE_x}\ \sfD_v
    \vf
    \big)_{\sfH_\Omega}
    +
    \big(\vv, 
    \big( 
    \overline{w\sfB_z}\ \sfD_v
    \big)\vf
    \big)_{\sfH_\Omega},
  \] 
  where
  \[
    \begin{aligned}
      \overline{\sfE_x} &= \operatorname{diag}(-\bE_x) \otimes \sfI_{m_{v}} \otimes \sfI_{m_{w}}, \\
      \overline{w\sfB_z} &= \operatorname{diag}(-\bB_z) \otimes \sfI_{m_v} \otimes \operatorname{diag}(\bw).
    \end{aligned}  \]

  Now, integrating by parts the last equality and using definitions of charge and current densities from \eqref{eq:disc:rho:J}, we obtain
  \[
    \big(\vv_f, 
    \overline{\sfA}_v \sfD_v
    \vU
    \big)_{\sfH_\Omega}
    =
    \big( \vE_x, \brho \big)_{\sfH_{\Omega_\bx}} 
    + 
    \big( \vB_z, \vJ_y \big)_{\sfH_{\Omega_\bx}},
  \]
  where $\vE_x$, $\vB_z$ and $\vJ_y$ are vectors containing the values of $E_x$, $B_z$ and $J_y$ on the nodal points. 

  Following a similar argument, we obtain 
  \[
    \big(\vw_f, 
    \overline{\sfA}_w \sfD_w
    \vU
    \big)_{\sfH_\Omega}
    =
    \big( \vE_y, \brho \big)_{\sfH_{\Omega_\bx}} 
    - 
    \big( \vB_z, \vJ_x \big)_{\sfH_{\Omega_\bx}}.
  \]

  Collecting all terms, 
  \begin{equation}\label{eq:mom:rel}
    \p_t 
    \begin{pmatrix}
      (\vv, \vf)_{\sfH_\Omega} \\
      (\vw, \vf)_{\sfH_\Omega}
    \end{pmatrix}
    =
    \begin{pmatrix}
      \big( \vE_x, \brho \big)_{\sfH_{\Omega_\bx}} 
      + 
      \big( \vB_z, \vJ_y \big)_{\sfH_{\Omega_\bx}} \\
      \big( \vE_y, \brho \big)_{\sfH_{\Omega_\bx}} 
      - 
      \big( \vB_z, \vJ_x \big)_{\sfH_{\Omega_\bx}}
    \end{pmatrix}.
  \end{equation}

  Using the Maxwell equations and Gauss' laws, we can repeat as in \eqref{eq:elec:mom} to complete the momentum conservation. 
\end{proof}

\begin{theorem}\label{thm:semi:l2}
  The semi-discrete upwind SBP approximation of the Vlasov--Maxwell equation
  \eqref{eq:upwind:sbp_2d2v} satisfies the discrete  $L^2$ stability estimate for the distribution function:
  \[
    \partial_t \|\vf\|_{\sfH_\Omega}^2 \le 0.
  \]
\end{theorem}

\begin{proof}
  Define
  \[
    \vf_f
    =
    \vf\otimes (1,0,0,0)^\ttop.
  \]

  Multiplying \eqref{eq:upwind:sbp_2d2v:short} by $\vf_f$ with respect to the discrete inner product on $\sfH_\Omega$ gives
  \[
    (\vf_f,\p_t\vU)_{\sfH_\Omega}
    =
    (\vf_f,\calA\vU)_{\sfH_\Omega}
    +
    (\vf_f,\calD\vU)_{\sfH_\Omega}.
  \]

  Using the SBP property and periodic boundary conditions, the advection contribution is skew-symmetric and therefore vanishes:
  \[
    (\vf_f,\calA\vU)_{\sfH_\Omega}=0.
  \]

  For the upwind stabilization terms, using \eqref{eq:upwind_identity}, we obtain
  \[
    (\vf_f,\calD\vU)_{\sfH_\Omega}
    \le 0,
  \]
  since $\sfS_z$ is negative semidefinite for each $z\in\{x,y,v,w\}$.

  Moreover,
  \[
    (\vf_f,\p_t\vU)_{\sfH_\Omega}
    =
    \frac12
    \p_t
    \|\vf\|_{\sfH_\Omega}^2.
  \]

  Combining these identities yields
  \[
    \p_t
    \|\vf\|_{\sfH_\Omega}^2
    \le 0,
  \]
  which proves the result.
\end{proof}

\begin{remark}[Total energy dissipation]
  Using an argument similar to that in the previous theorems, one can show that the total energy is conserved for the central-difference SBP scheme \eqref{eq:sbp_2d2v}. However, the upwind stabilization terms in \eqref{eq:upwind:sbp_2d2v} introduce dissipation and therefore cause the total energy to decrease.
\end{remark}

\begin{remark}[Discrete Gauss law]\label{rem:gauss}
  The reduced Maxwell system is supplemented by Gauss' law
  \[
    \partial_x E_x+\partial_y E_y=\rho .
  \]
  In the semi-discrete SBP formulation, we impose the corresponding discrete constraint
  \[
    \sfD_x \vE_x+\sfD_y \vE_y=\brho .
  \]
  This constraint is propagated by the semi-discrete equations. Indeed, applying
  \(\sfD_x\) to the equation for \(\vE_x\) and \(\sfD_y\) to the equation for
  \(\vE_y\), and using the commutativity of tensor-product derivative operators, gives
  \[
    \partial_t(\sfD_x\vE_x+\sfD_y\vE_y)
    =
    -(\sfD_x\vJ_x+\sfD_y\vJ_y).
  \]
  On the other hand, taking the velocity integral of the semi-discrete Vlasov equation gives the discrete continuity equation
  \[
    \partial_t\brho+\sfD_x\vJ_x+\sfD_y\vJ_y=0.
  \]
  Therefore,
  \[
    \partial_t(\sfD_x\vE_x+\sfD_y\vE_y-\brho)=0.
  \]
  Thus, if the discrete Gauss law is satisfied initially, it remains satisfied for all time at the semi-discrete level.
\end{remark}

\section{Fully discrete approximation}
\label{Sec:fully-discrete}
In this section, we present a fully discrete approximation of the Vlasov--Maxwell system. In our numerical validations, we employ the fourth order, five-stage
strong stability preserving explicit Runge--Kutta (SSP-RK(5,4)) method of \cite{Kraaijevanger_1991} in time. Since SSP methods are by construction convex combinations of Forward Euler steps, it is sufficient to present the method and establish conservation of mass and momentum for Forward Euler.

\subsection{Forward Euler time discretization}
Let us denote the current time by $t^n \ge 0$, the current time-step by $\tau_n = t^{n+1} - t^n$, and discretize the time interval by $0=t^0 < t^1 < \ldots < t^N = T$. Let $\vU^n \approx \vU(t^n)$ be the finite difference approximation of the solution at time $t^n$. The Forward Euler discretization of \eqref{eq:upwind:sbp_2d2v:short} is given by: 
\begin{equation}\label{eq:fe}
  \frac{\vU^{n+1}-\vU^{n}}{\tau_n} = \calA \vU^n + \calD \vU^n + \overline{\sfF}^n, \quad n=0,1,\ldots, N.
\end{equation}

The time-step is computed by the following CFL condition
\begin{equation}\label{eq:cfl}
  \tau_n = \textrm{cfl} \min_{z\in \{x,y,v,w\}} \frac{h_z}{\alpha_{\overline{\sfA}_z}},
\end{equation}
where $\textrm{cfl}>0$ is the CFL number. 

Let us define the total mass and momentum at time $t^n$ by:
\[
  M^n := (\vone,\vf^n)_{\sfH_\Omega}, \mbox{ and } 
  \bP^n = 
  \begin{pmatrix}
    P_x^n \\
    P_y^n
  \end{pmatrix}
  :=
  \begin{pmatrix}
    (\vv,\vf^n)_{\sfH_\Omega} + (E_y^n,B_z^n)_{\sfH_{\Omega_x}} \\
    (\vw,\vf^n)_{\sfH_\Omega} - (E_x^n,B_z^n)_{\sfH_{\Omega_x}}
  \end{pmatrix}.
\]

\begin{theorem}
  Under the Forward Euler discretization \eqref{eq:fe}, the fully discrete upwind SBP scheme
  has the following conservation properties:
  \[
    M^{n+1}=M^n,
    \qquad
    \bP^{n+1}=\bP^n + \calO(\tau_n^2).
  \]
\end{theorem}

\begin{proof}
  \noindent {\bf Mass conservation}. 
  Multiplying \eqref{eq:fe} by $\vone_f$ with respect to the discrete inner product on $\sfH_\Omega$, we obtain
  \[
    \frac{1}{\tau_n}(\vone_f, (\vU^{n+1} -\vU^{n}))_{\sfH_\Omega} 
    =   
    (\vone_f, \calA \vU^n + \calD \vU^n + \overline{\sfF}^n)_{\sfH_\Omega}.
  \]
  Again applying integration by parts together with periodic boundary conditions gives that the right hand side is zero. Therefore,
  \[
    (\vone_f, \vU^{n+1})_{\sfH_\Omega}
    = 
    (\vone_f, \vU^{n})_{\sfH_\Omega},
  \]
  or
  \[
    (\vone, \vf^{n+1})_{\sfH_\Omega}
    = 
    (\vone, \vf^{n})_{\sfH_\Omega}.
  \]

  \noindent {\bf Momentum conservation}. 
  We perform the proof for the first component of the total momentum. We have
  \[
    \begin{aligned}
      P^{n+1}_x - P^n_x 
      &= 
        \Big( 
        (\vv,\vf^{n+1})_{\sfH_\Omega} + (\vE_y^{n+1}, \vB_z^{n+1})_{\sfH_{\Omega_x}} 
        \Big)
        -
        \Big(
        (\vv,\vf^n)_{\sfH_\Omega} + (\vE_y^n, \vB_z^n)_{\sfH_{\Omega_x}}
        \Big) \\
      &=
        \Big( 
        (\vv,\vf^{n+1})_{\sfH_\Omega} -  (\vv,\vf^{n})_{\sfH_\Omega}
        \Big)
        +
        \Big(
        (\vE_y^{n+1}, \vB_z^{n+1})_{\sfH_{\Omega_x}} - (\vE_y^{n}, \vB_z^{n})_{\sfH_{\Omega_x}}
        \Big) \\
      &:=
        I_1 + I_2.
    \end{aligned}
  \]

  Multiplying \eqref{eq:fe} by $\vv_f$ with respect to the discrete inner product on $\sfH_\Omega$, we obtain:
  \[
    (\vv_f,\vU^{n+1})_{\sfH_\Omega} - (\vv_f,\vU^{n})_{\sfH_\Omega}
    =
    \tau_n
    (\vv_f, \calA \vU^n + \calD \vU^n + \overline{\sfF}^n).
  \]
  We have established earlier that the right-hand-side of this equality becomes the first row of \eqref{eq:mom:rel}. And the left-hand-side can be simplified:
  \[
    (\vv,\vf^{n+1})_{\sfH_\Omega} - (\vv,\vf^{n})_{\sfH_\Omega}
    =
    \tau_n 
    \Big(
    \big( \vE^n_x, \brho^n \big)_{\sfH_{\Omega_\bx}} 
    + 
    \big( \vB^n_z, \vJ^n_y \big)_{\sfH_{\Omega_\bx}}
    \Big),
  \]
  thus
  \[
    I_1 
    = 
    \tau_n 
    \Big(
    \big( \vE^n_x, \brho^n \big)_{\sfH_{\Omega_\bx}} 
    + 
    \big( \vB^n_z, \vJ^n_y \big)_{\sfH_{\Omega_\bx}}
    \Big).
  \]

  For the second term $I_2$, we first discretize the Maxwell's equations in \eqref{eq:2d2v} using Forward Euler method:
  \[
    \begin{aligned}
      \vE^{n+1}_y &= \vE^{n}_y + \tau_n \big(-\sfD_x \vB^n_z - \vJ^n_y \big), \\
      \vB^{n+1}_z &= \vB^{n}_z + \tau_n \big( \sfD_y \vE^n_x - \sfD_x\vE^n_y \big). \\
    \end{aligned}  
  \]

  Therefore, we get
  \[
    \begin{aligned}
      I_2 
      &= 
        \big(
        \vE^{n}_y + \tau_n \big(-\sfD_x \vB^n_z - \vJ^n_y \big),\
        \vB^{n}_z + \tau_n \big( \sfD_y \vE^n_x - \sfD_x\vE^n_y \big)
        \big)_{\sfH_{\Omega_x}}
        - (\vE_y^{n}, \vB_z^{n})_{\sfH_{\Omega_x}} \\
      &=
        -\tau_n(\sfD_x\vB^n_z, \vB^n_z)_{\sfH_{\Omega_x}} 
        -\tau_n(\vJ^n_y, \vB^n_z)_{\sfH_{\Omega_x}} \\
      &\hspace{1in} +\tau_n(\vE^n_y, \sfD_y\vE^n_x)_{\sfH_{\Omega_x}}
        -\tau_n(\vE^n_y, \sfD_x\vE^n_y)_{\sfH_{\Omega_x}}
        +\calO(\tau_n^2).
    \end{aligned}
  \]

  The first and the last product terms vanish after integration by parts. The third term can be written
  \[
    \begin{aligned}
      \tau_n(\vE^n_y, \sfD_y\vE^n_x)_{\sfH_{\Omega_x}}
      &=
        -\tau_n(\sfD_y \vE^n_y, \vE^n_x)_{\sfH_{\Omega_x}} \\
      &=
        \tau_n(\brho^n, \vE^n_x)_{\sfH_{\Omega_x}} 
        - \tau_n(\sfD_x \vE^n_x, \vE^n_x)_{\sfH_{\Omega_x}} \\
      &=
        \tau_n(\brho^n, \vE^n_x)_{\sfH_{\Omega_x}},
    \end{aligned}
  \]
  where we performed integration by parts and used Gauss' law. 

  Now, collecting the remaining terms gives us
  \[
    I_2 
    = 
    -\tau_n(\vJ^n_y, \vB^n_z)_{\sfH_{\Omega_x}} 
    -\tau_n(\brho^n, \vE^n_x)_{\sfH_{\Omega_x}}
    + \calO(\tau_n^2).
  \]

  Finally, 
  \[
    P^{n+1}_x - P^n_x =  I_1 + I_2 = \calO(\tau_n^2).
  \] 

  In exactly the same way we obtain
  \[
    P^{n+1}_y - P^n_y =  \calO(\tau_n^2).
  \] 

  Thus, the Forward Euler method conserves discrete momentum up to truncation error. 
\end{proof}

\begin{remark}[Conservation of the $L^2$-norm and total energy]
  Although we observe numerically that the $L^2$-norm remains stable, a rigorous proof is challenging. In particular, the upwind dissipation may vanish when the solution is constant or nearly constant, which prevents establishing a uniform decay estimate. 

  Moreover, due to the added numerical stabilization, the total energy is no longer conserved and instead decays in time, which is also confirmed by the numerical experiments.
\end{remark}

\begin{remark}[Discrete Gauss law]
  The reduced Maxwell system is supplemented by Gauss' law
  \[
    \partial_x E_x+\partial_y E_y=\rho .
  \]
  In the semi-discrete SBP formulation, we impose the corresponding discrete constraint
  \[
    \sfD_x \vE_x+\sfD_y \vE_y=\brho .
  \]
  This constraint is propagated by the semi-discrete equations. Indeed, applying
  \(\sfD_x\) to the equation for \(\vE_x\) and \(\sfD_y\) to the equation for
  \(\vE_y\), and using the commutativity of tensor-product derivative operators, gives
  \[
    \partial_t(\sfD_x\vE_x+\sfD_y\vE_y)
    =
    -(\sfD_x\vJ_x+\sfD_y\vJ_y).
  \]
  On the other hand, taking the velocity integral of the semi-discrete Vlasov equation gives the discrete continuity equation
  \[
    \partial_t\brho+\sfD_x\vJ_x+\sfD_y\vJ_y=0.
  \]
  Therefore,
  \[
    \partial_t(\sfD_x\vE_x+\sfD_y\vE_y-\brho)=0.
  \]
  Thus, if the discrete Gauss law is satisfied initially, it remains satisfied for all time at the semi-discrete level.
\end{remark}

\section{GPU implementation}
\label{Sec:gpu}
One of the advantages of the finite difference methods as presented in this paper is a diagonal mass matrix, in contrast to finite element approximations of Vlasov--Maxwell presented in \citep{Kormann_2025}. However, assembling and storing the coefficient matrices as well as SBP operators in \eqref{eq:sbp_2d2v} is expensive and impractical in high dimensions. Instead, we use a matrix-free GPU implementation in this work, where the SBP operators are applied to the variables directly through their local stencil coefficients. Below, we present details of our GPU implementation.

We start by storing the distribution function as a one-dimensional array corresponding to the flattened four-dimensional grid
\[
  (i,j,k,l) \longleftrightarrow x_i,y_j,v_k,w_l .
\]
The electromagnetic fields $E_x$, $B_z$, and $E_y$ depend only on the spatial variables and are therefore stored on the two-dimensional grid $(x_i,y_j)$. This separation substantially reduces memory usage for the field variables.

The finite difference operators are implemented in stencil form. For example, a derivative in the $x$-direction is evaluated as
\[
  (\sfD_z f)_{i,j,k,l}
  =
  \sum_{r=-p}^{p} c_r^{(z)}
  f_{i+r,j,k,l},
\]
where periodic indexing is used. Analogous stencil evaluations are used in the $y$, $v$, and $w$ directions. The stencil coefficients are copied to CUDA constant memory, allowing all GPU threads to access the same coefficients efficiently. Both central and upwind SBP stencil coefficients are stored in this way.

The right-hand side of the Vlasov equation is evaluated by assigning GPU threads to phase space grid points. The Lorentz-force terms are then evaluated locally using the field values at the corresponding spatial point. Thus, the semi-discrete Vlasov equation is evaluated without forming the global matrices
\[
  \overline{\sfA}_x\sfD_x,\quad
  \overline{\sfA}_y\sfD_y,\quad
  \overline{\sfA}_v\sfD_v,\quad
  \overline{\sfA}_w\sfD_w .
\]
Similarly, the corresponding upwind operator $\sfD_z$ is evaluated as
\[
  (\sfD_{z,\pm} f)_{i,j,k,l}
  =
  \sum_{r=-p}^{p} c_r^{(z,\pm)}
  f_{i+r,j,k,l},
\]
and then the upwind dissipation term is computed pointwise as
\[
  k_s \alpha_{\overline{\sfA}_z} \sfH^{-1}_z \sfS_z f
  =
  \frac12 k_s \alpha_{\overline{\sfA}_z} 
  \big(
  \sfD_{z,+} - \sfD_{z,-}
  \big)f,
\]
without ever constructing the operators $\sfH_z$ and $\sfS_z$.

The current density is computed simultaneously with the Vlasov right-hand side. For each spatial point, the moments
\[
  J_{x,i,j}%
  =
  \sum_{k,l} \omega_k \omega_l \, v_k \, f_{i,j,k,l},
  \quad
  J_{y,i,j}%
  =
  \sum_{k,l} \omega_k \omega_l \, w_l \, f_{i,j,k,l},
\]
are approximated by quadrature over the velocity grid. Since many phase space threads contribute to the same spatial current value, these contributions are accumulated using atomic additions on the GPU.

The Maxwell equations are advanced by a separate GPU kernel on the two-dimensional spatial grid. In this kernel, the derivatives of $B_z$, $E_x$, and $E_y$ are computed using the same stencil-based finite difference operators. The update has the form of the Maxwell equations \eqref{eq:2d2v} for the variables $E_x, B_z$ and $E_y$. The implementation uses periodic indexing in both spatial directions.

Time integration is performed using the fourth-order, five-stage strong stability
preserving explicit Runge--Kutta method. At each Runge--Kutta stage, the current density is recomputed from the stage value of $f$, and the Maxwell right-hand side is evaluated using the corresponding stage fields. Intermediate Runge--Kutta arrays are stored on the device, and the final update is performed by separate GPU kernels for the distribution function and the electromagnetic fields.

The implementation is therefore fully matrix-free: only the solution arrays, field arrays, velocity grids, current densities, and Runge--Kutta stage arrays are stored. The differentiation matrices are never assembled. This is essential for high-dimensional simulations, where explicit matrix storage would be prohibitively expensive. The tensor-product SBP structure is used only through local stencil applications, making the method well suited for GPU acceleration.

\section{Numerical experiments}
\label{sec:results}
In this section, the results from solving the 2D2V Vlasov--Maxwell system of equations are presented,  using benchmark-problems. The time-step for all simulations is computed by \eqref{eq:cfl}. 
The value of $\text{CFL}$ is presented separately for each result below.
Unless stated otherwise, all simulations are performed with periodic FD operators of order 6, using $k_s=1$ for the stabilization
and single precision on the GPU.
As for GPU architecture, we run on one NVIDIA L40 from the UPPMAX Pelle cluster. 

\subsection{Verification of solver}
\label{subsec:mms}
We do a convergence study of the 2D2V solver using the method of manufactured solutions (MMS). We choose the distribution and fields to be
\begin{equation}
  \begin{aligned}
    f(\bx , \bv , t) &= A_f\sin{(k_xx)}\sin{(k_yy)}\sin{(k_vv)}\sin{(k_ww)}\cos{(\omega t)}, \\
    E_x(\bx , t) &= A_{E_x}\sin{(k_xx)}\sin{(k_yy)}\cos{(\omega t)}, \\
    B_z(\bx , t) &= A_{B_z}\sin{(k_xx)}\sin{(k_yy)}\cos{(\omega t)}, \\
    E_y(\bx , t) &= A_{E_y}\sin{(k_xx)}\sin{(k_yy)}\cos{(\omega t)},
  \end{aligned}
  \label{eq:MMS_solution}
\end{equation}
where we in the proceeding measurements set all constants $A_z=0.1$, $\omega=1$ and $k_z=2\pi / L_z$, for $z \in \set{x, y, v, w}$. The length $L_z$ is the distance between endpoints for said dimension, which is used to make the solution in \eqref{eq:MMS_solution} periodic in phase space. 
We set $L_z=1$ for all dimensions.
Furthermore, we set $\|\bB_{\text{ext}}\|=1$ and $\varepsilon = 1$. In this periodic case, where we do not use closures from the SBP framework, the discrete $l_2$-error for the discrete solution vector $\Theta$ is calculated as
\begin{equation}\label{eq:l2_error}
  \norm{\Theta}_{l_2} = \sqrt{ \frac{1}{N_{\text{dof}}} \sum_{i=0}^{N_{\text{dof}}-1}(\Theta_i - \bar\Theta_i)^2},
\end{equation}
where $N_{\text{dof}}$ is the number of degrees of freedom and $\bar \Theta_i$ denotes the exact solution at coordinate $i$. CFL is set to 0.2
and double precision is used on the GPU.
\begin{figure}[h] %
  \centering
  \begin{subfigure}[b]{0.45\linewidth}
    \centering
    \includegraphics[width=\linewidth]{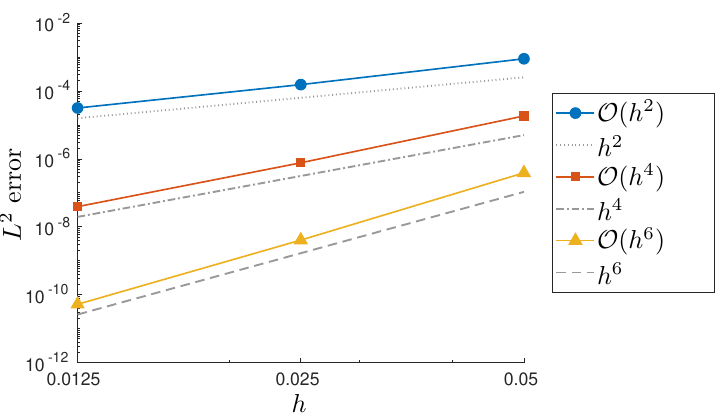}
    \caption{$f$}
  \end{subfigure}
  \hfill
  \begin{subfigure}[b]{0.45\linewidth}
    \centering
    \includegraphics[width=\linewidth]{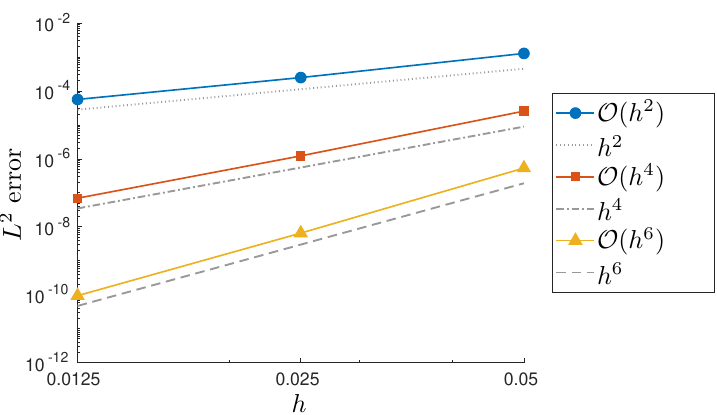}
    \caption{$E_x$}
  \end{subfigure}
  \\
  \vspace{1em}
  \begin{subfigure}[b]{0.45\linewidth}
    \centering
    \includegraphics[width=\linewidth]{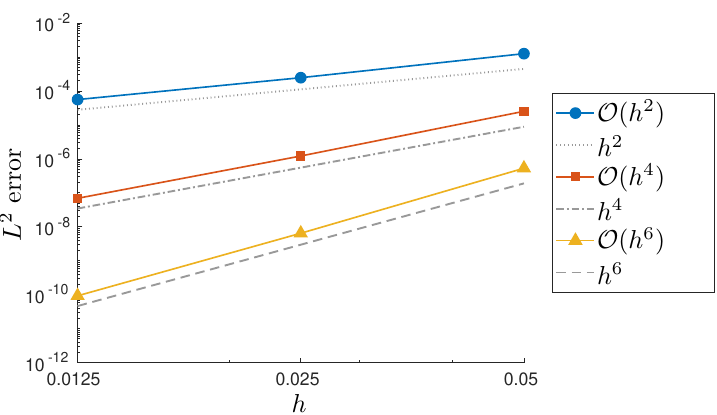}
    \caption{$E_y$}
  \end{subfigure}
  \hfill
  \begin{subfigure}[b]{0.45\linewidth}
    \centering
    \includegraphics[width=\linewidth]{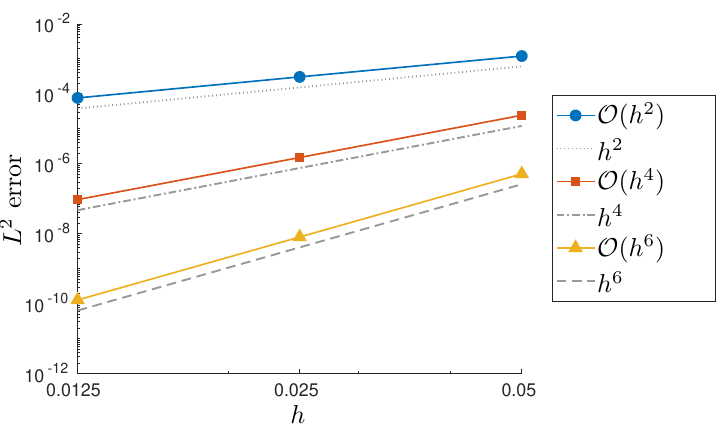}
    \caption{$B_z$}
  \end{subfigure}
  
  \caption{MMS convergence test: We run the simulation to $T=1$ for $m=20$, $m=40$ and $m=80$, where $m_x=m_y=m_v=m_w=m$.}
  \label{fig:2D2V_MMS_convergence}
\end{figure}

\begin{table}[h]
  \centering
  \small
  \begin{tabular}{c c cc cc cc cc}
    \toprule
    & & \multicolumn{2}{c}{$f$}
      & \multicolumn{2}{c}{$E_x$}
      & \multicolumn{2}{c}{$E_y$}
      & \multicolumn{2}{c}{$B_z$} \\
    \cmidrule(lr){3-4} \cmidrule(lr){5-6} \cmidrule(lr){7-8} \cmidrule(lr){9-10}
    Order & $m$ & $L^2$ & $p$ & $L^2$ & $p$ & $L^2$ & $p$ & $L^2$ & $p$ \\
    \midrule
    \multirow{3}{*}{2nd}
      & $20$ & 8.84E-04 & --   & 1.27E-03 & --   & 1.25E-03 & --   & 1.20E-03 & --   \\
      & $40$ & 1.53E-04 & 2.53 & 2.47E-04 & 2.36 & 2.45E-04 & 2.35 & 3.04E-04 & 1.98 \\
      & $80$ & 3.14E-05 & 2.28 & 5.61E-05 & 2.14 & 5.56E-05 & 2.14 & 7.60E-05 & 2.00 \\
    \addlinespace
    \multirow{3}{*}{4th}
      & $20$ & 1.83E-05 & --   & 2.54E-05 & --   & 2.51E-05 & --   & 2.43E-05 & --   \\
      & $40$ & 7.63E-07 & 4.58 & 1.22E-06 & 4.38 & 1.21E-06 & 4.38 & 1.50E-06 & 4.01 \\
      & $80$ & 3.90E-08 & 4.29 & 6.92E-08 & 4.15 & 6.87E-08 & 4.14 & 9.38E-08 & 4.00 \\
    \addlinespace
    \multirow{3}{*}{6th}
      & $20$ & 3.86E-07 & --   & 5.34E-07 & --   & 5.28E-07 & --   & 5.10E-07 & --   \\
      & $40$ & 4.04E-09 & 6.58 & 6.48E-09 & 6.36 & 6.42E-09 & 6.36 & 7.98E-09 & 6.00 \\
      & $80$ & 5.19E-11 & 6.28 & 9.30E-11 & 6.12 & 9.23E-11 & 6.12 & 1.26E-10 & 5.98 \\
    \bottomrule
  \end{tabular}
  \caption{Convergence under grid refinement. $L^2$ errors and observed
    convergence rates $p$ for the distribution function $f$ and the field
    components $E_x$, $E_y$ and $B_z$, using 2nd-, 4th- and 6th-order SBP
    operators.}
  \label{tab:mms_convergence}
\end{table}

Figure \ref{fig:2D2V_MMS_convergence} and Table \ref{tab:mms_convergence} show the expected convergence rates
for all variables.

\subsection{Diocotron instability}
We now consider a test problem for the 2D2V Vlasov--Maxwell system,
namely the diocotron instability as presented in \cite{Gu_2022}.
The instability manifests as the formation of vortices over time and occurs when two sheets of charges slip past each other.
We will include an external magnetic field in our model, which gives
us the following version of Vlasov's equation:
\begin{equation}
  \varepsilon f_t + \bv  \cdot \GRAD_x f + \frac{q}{m}\left( \bE  + \bv  \times \left( \bB  + \frac{1}{\varepsilon}\bB_{\text{ext}}\right) \right) \cdot \GRAD_v f = 0,
  \label{eq:vlasov_Bext}
\end{equation}
where $\bB $ is the self-consistent magnetic field and $\bB_{\text{ext}}$ is an external magnetic field. $1/\varepsilon$ determines
the strength of the external magnetic field. The factor $\varepsilon$
in front of $f_t$ is to speed up the evolution of the density function.
For consistency, the Maxwell equations are scaled accordingly:
\begin{equation}
  \begin{aligned}
    \GRAD_{\bx } \times \bE  &= -\varepsilon\, \partial_t \bB , \\
    \GRAD_{\bx } \times \bB  &= \ \varepsilon\, \partial_t \bE  + \bJ , \\
    \GRAD_{\bx } \cdot \bE  &= \rho - \rho_0, \\
    \GRAD_{\bx } \cdot \bB  &= 0,
                              \label{eq:maxwell_Bext}
  \end{aligned}
\end{equation}

where $\rho_0$ denotes the uniform background charge density.
We set the initial distribution as
\begin{equation}
  f(\bx , \bv ) = \frac{d(\bx )}{2\pi} \exp\left(-\frac{\|\bv \|^2}{2}\right), \quad \bx  = (x,y) \in \mathbb{R}^2,
  \label{eq:diocotron_ic}
\end{equation}
where the initial density is
\begin{equation}\label{eq:diocotron_density}
  d(\bx ) = 
  \begin{cases} 
    (1 + \beta \cos(\ell\theta)) \exp(-4(\|\bx \| - 6.5)^2) & \text{if } r^- \leq \|\bx \| \leq r^+, \\ 
    0 & \text{otherwise},
  \end{cases}
\end{equation}
with $\theta = \arctan(y/x)$ and $\ell$ the number of vortices.
In the proceeding simulations, we use $\beta = 0.2$, $r^-=5$ and
$r^+=8$. We set the external magnetic field to be constant in the $z$-axis
and zero elsewhere, i.e.\ $\bB_{\text{ext}} = (0,0,1)$. We test
with $\varepsilon=1$ and $\varepsilon=0.1$. The solutions in Figure \ref{fig:diocotron_eps1} and \ref{fig:diocotron_eps10} are visualized by integrating over the velocity dimensions.

\begin{figure}[h] %
  \centering
  \begin{subfigure}[b]{0.4\linewidth}
    \centering
    \includegraphics[width=\linewidth]{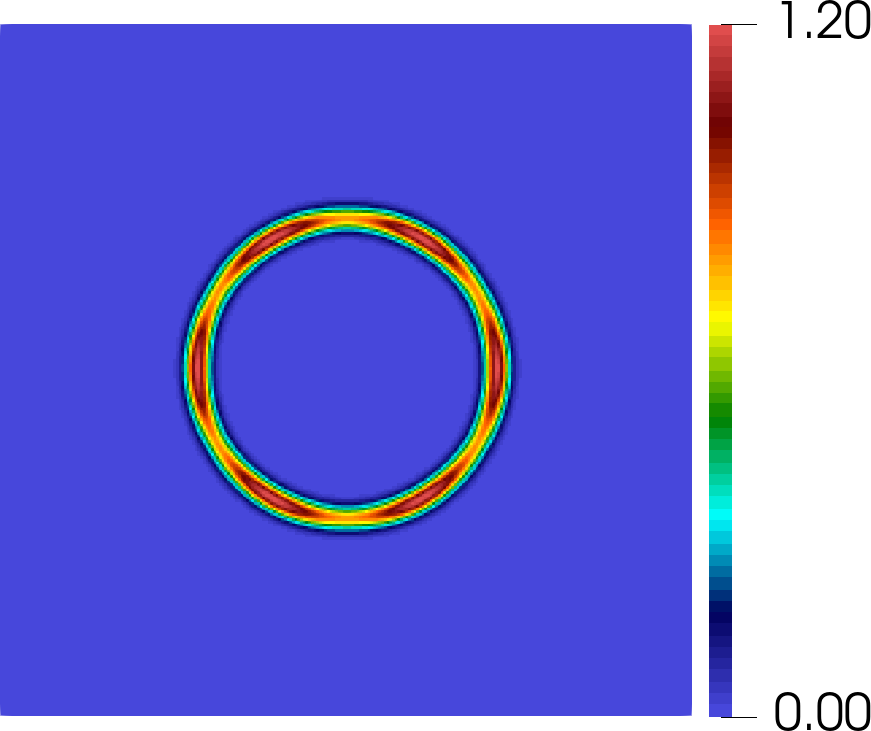}
    \caption{$t=0$}
  \end{subfigure}
  \hspace{0.2in}
  \begin{subfigure}[b]{0.4\linewidth}
    \centering
    \includegraphics[width=\linewidth]{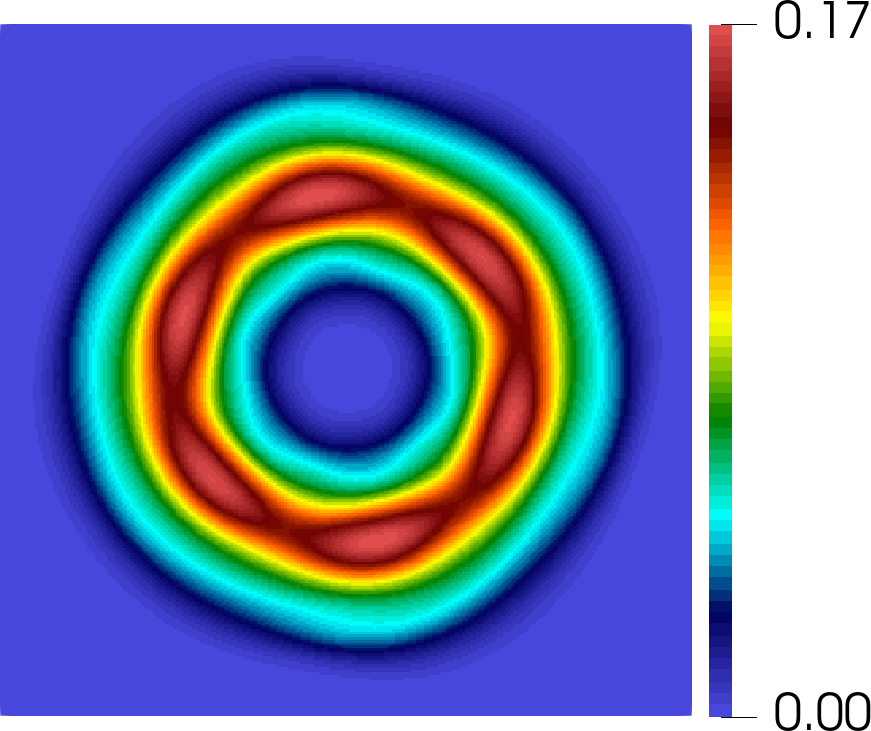}
    \caption{$t=10$}
  \end{subfigure}
  \\
  \begin{subfigure}[b]{0.4\linewidth}
    \centering
    \includegraphics[width=\linewidth]{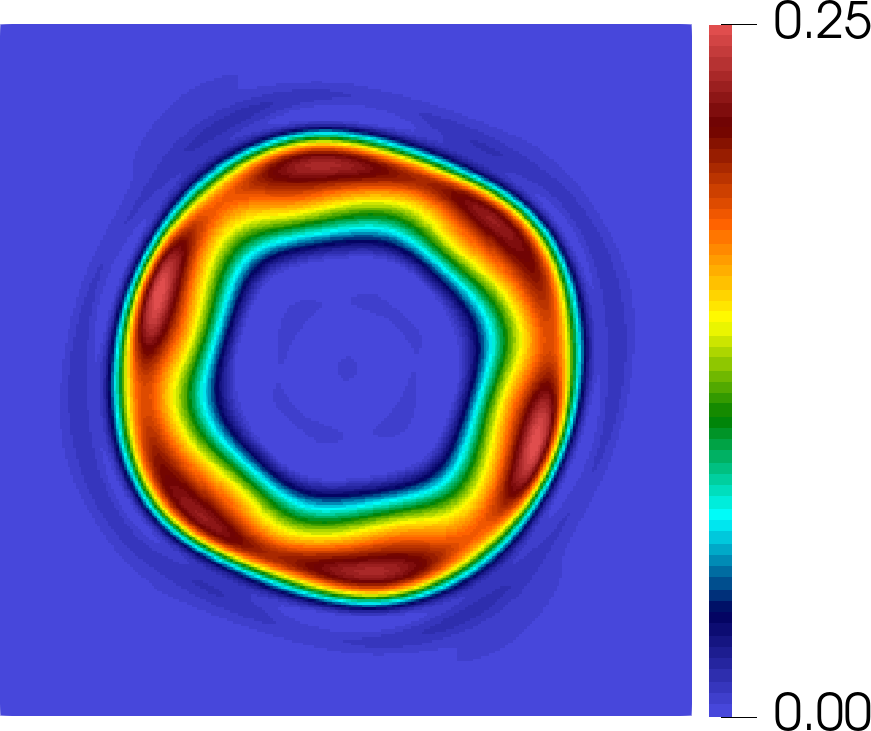}
    \caption{$t=15$}
  \end{subfigure}
  \hspace{0.2in}
  \begin{subfigure}[b]{0.4\linewidth}
    \centering
    \includegraphics[width=\linewidth]{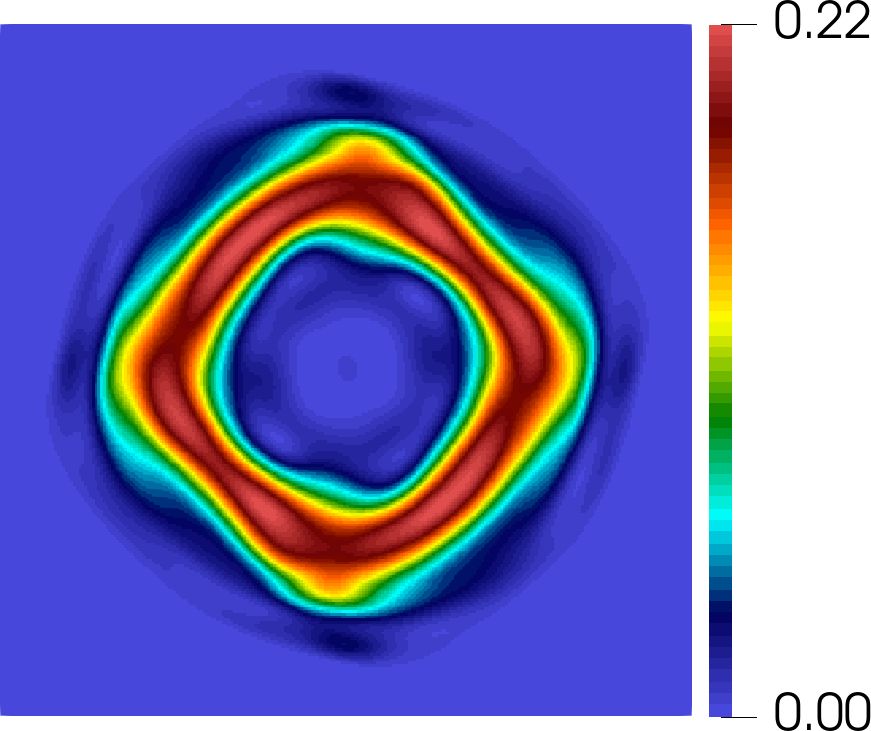}
    \caption{$t=30$}
  \end{subfigure}
  
  \caption{Diocotron instability: $\varepsilon=1$, $m_x=m_y=256$ and $m_v=m_w=80$ solution of $f(x, y, t)$. The horizontal and vertical axes represent position $x$ and $y$, respectively.}
  \label{fig:diocotron_eps1}
\end{figure}

In Figure \ref{fig:diocotron_eps1}, we see that the plasma is not
well confined and no diocotron instabilities are present, which is expected with a weak external magnetic field. Therefore, no vortices are seen,
but for $\ell=6$ there are six clear clusters that move in time.
CFL is set to 0.5, which yields the time-step size $\Delta t = 0.0073$, where each time-step takes $0.25$ seconds on the GPU.
Therefore, running the simulation with time $T=30$ takes about 17 minutes.

\begin{figure}[h] %
  \centering
  \begin{subfigure}[b]{0.4\linewidth}
    \centering
    \includegraphics[width=\linewidth]{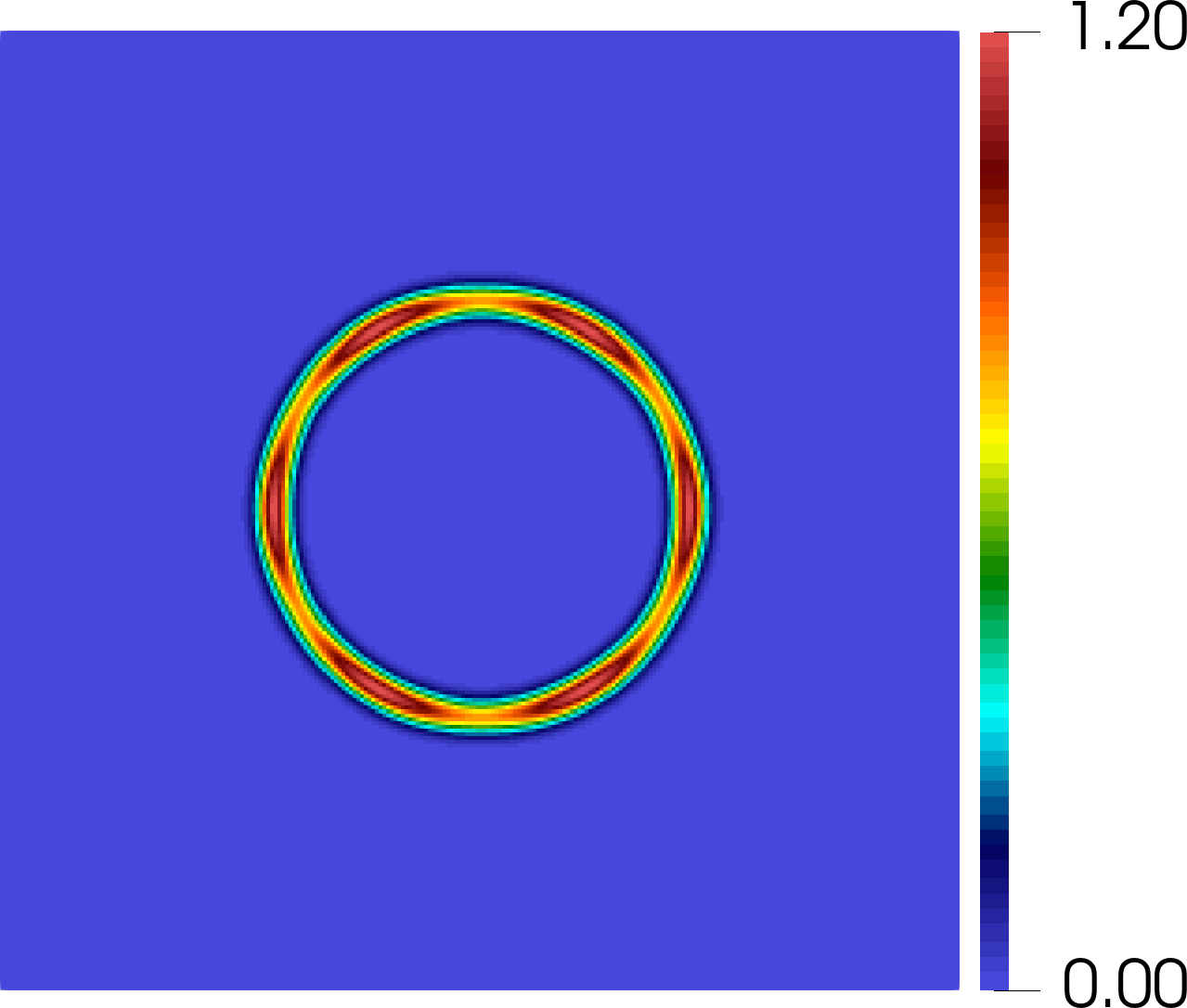}
    \caption{$t=0$}
  \end{subfigure}
  \hspace{0.2in}
  \begin{subfigure}[b]{0.4\linewidth}
    \centering
    \includegraphics[width=\linewidth]{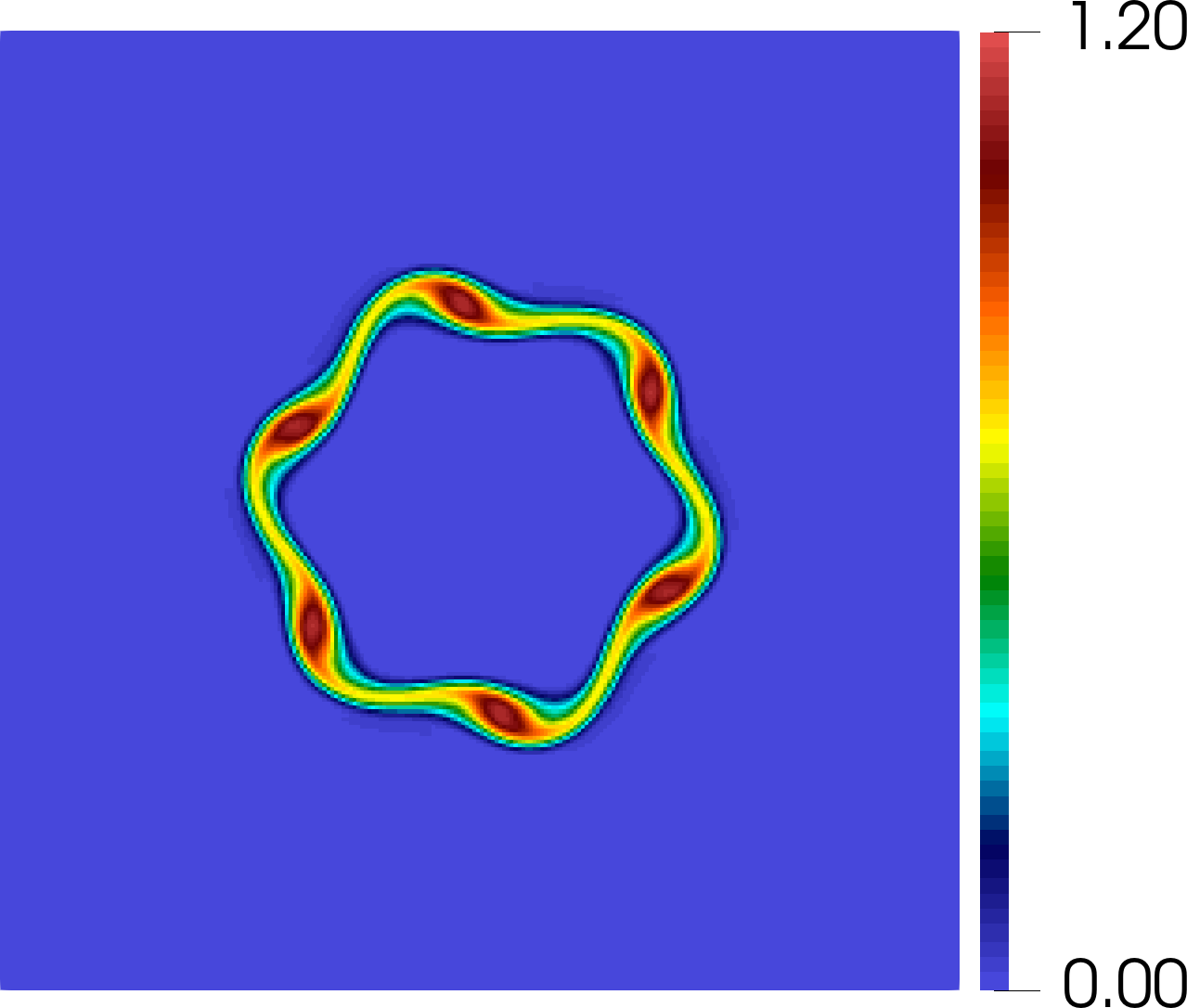}
    \caption{$t=10$}
  \end{subfigure}
  \\
  \begin{subfigure}[b]{0.4\linewidth}
    \centering
    \includegraphics[width=\linewidth]{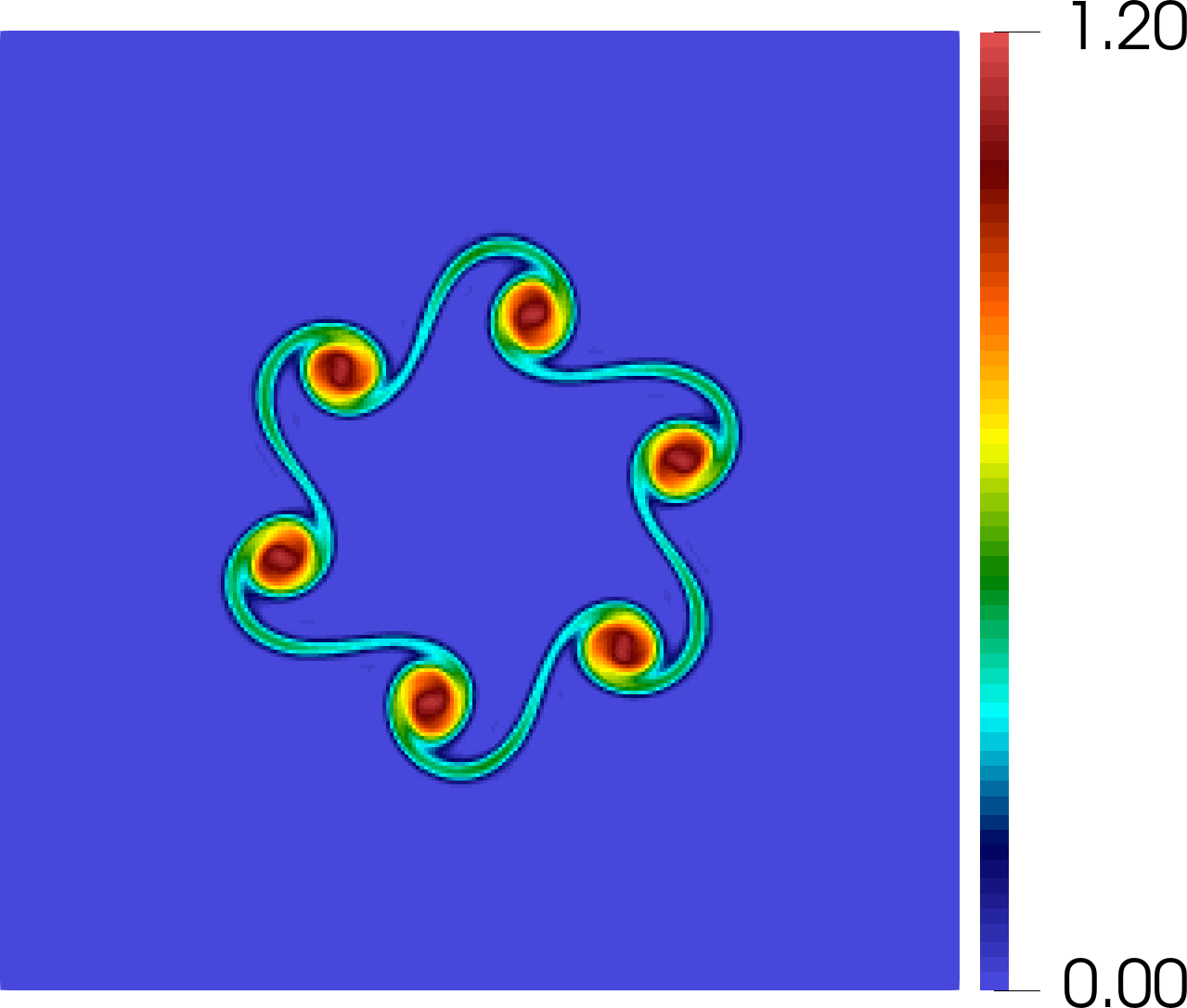}
    \caption{$t=20$}
  \end{subfigure}
  \hspace{0.2in}
  \begin{subfigure}[b]{0.4\linewidth}
    \centering
    \includegraphics[width=\linewidth]{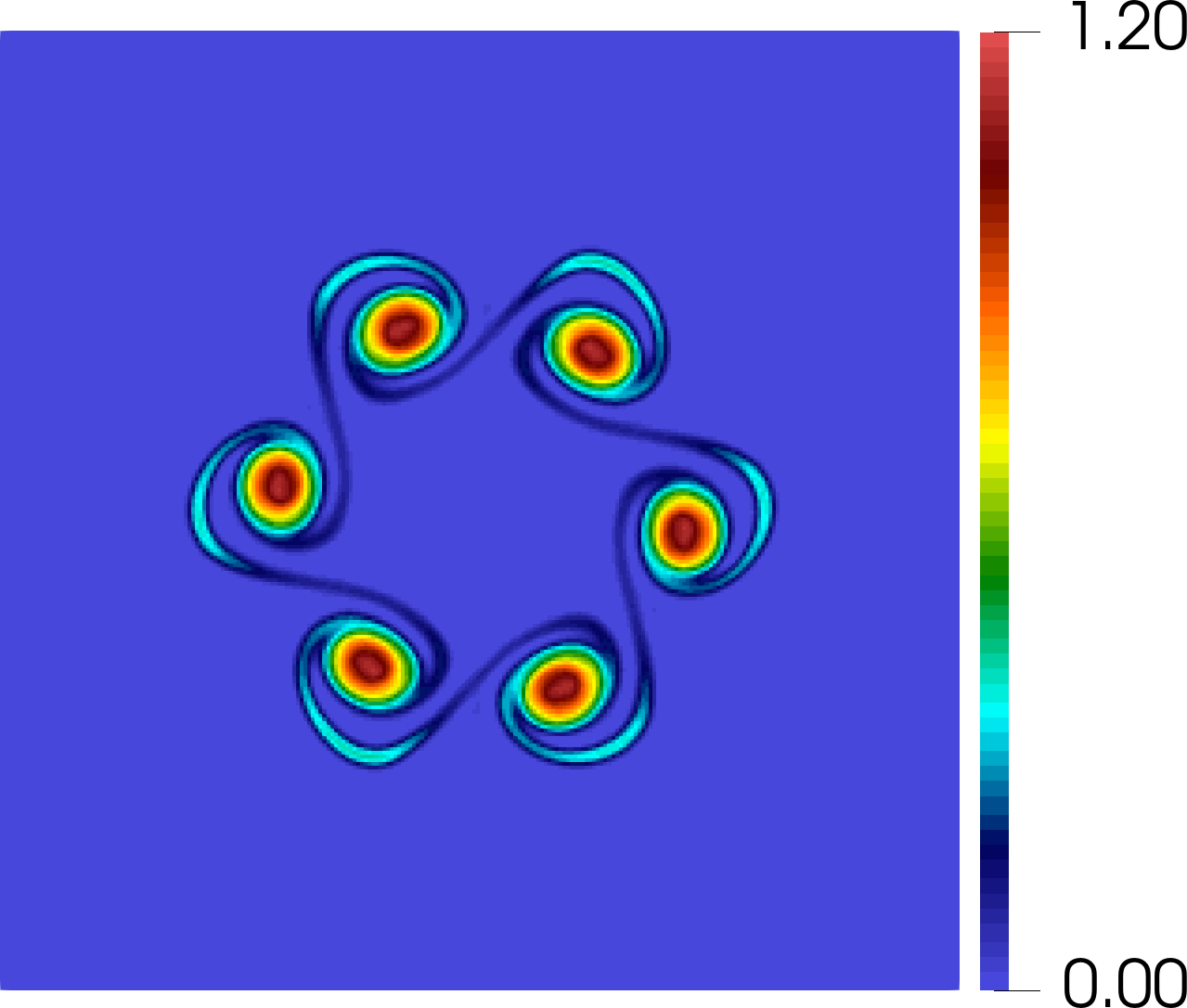}
    \caption{$t=30$}
  \end{subfigure}
  
  \caption{Diocotron instability: $\varepsilon=0.1$, $m_x=m_y=256$ and $m_v=m_w=80$ solution of $f(x, y, t)$. The horizontal and vertical axes represent position $x$ and $y$, respectively.}
  \label{fig:diocotron_eps10}
\end{figure}
With a stronger magnetic field,
the plasma is now well confined and $\ell=6$ vortices form from the
diocotron instability, seen in Figure \ref{fig:diocotron_eps10}.
For the $\varepsilon=0.1$ case, the simulation takes much longer to run compared to the $\varepsilon=1$ case, because the system becomes
stiffer with increased strength of the magnetic field. %
Using $\text{CFL}=0.5$, the time-step becomes $\Delta t = 0.000123633$, $0.25$ seconds per time-step and about 17 hours for the whole simulation.

To conclude this benchmark, we verify Gauss' law by numerically computing and plotting the divergence residual,
\begin{equation}
  R(\bx, t) = \nabla \cdot \bE - (\rho - \rho_0).
  \label{eq:guass_law_residual}
\end{equation}
\begin{figure}[h]
  \centering
  \includegraphics[width=0.5\linewidth]{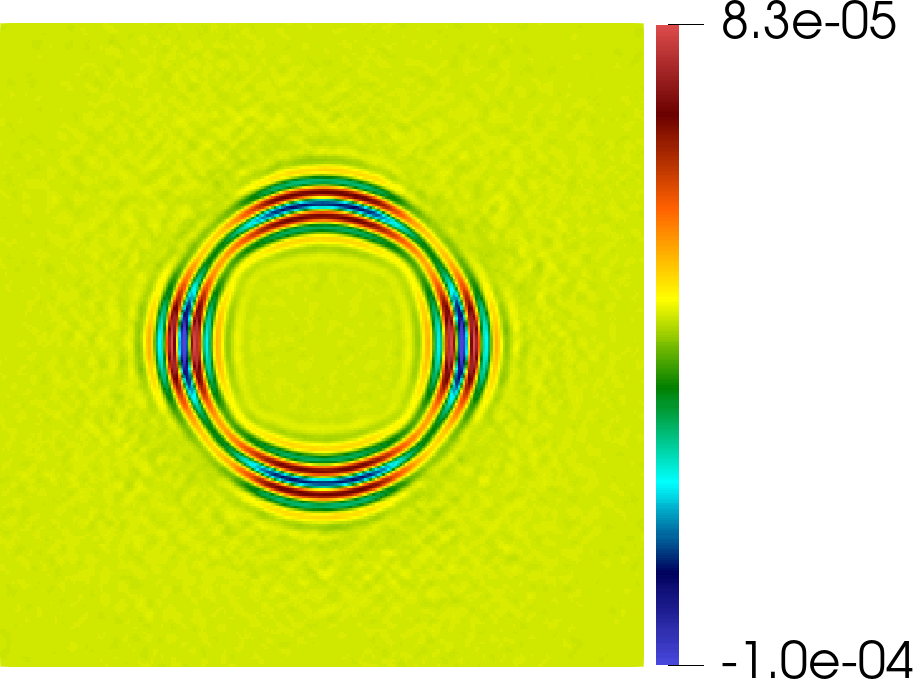}
  \caption{Gauss' law verification: $R(\bx, t)$ at $t=30$. The small residual seen in the figure is 
  SBP FD discretization error that stems from the solution to Poisson's equation. This error decays slightly in time
  due to numerical dissipation.}
  \label{fig:gauss_law_verification}
\end{figure}
We see in Figure \ref{fig:gauss_law_verification} that the residual stays small in time. It is largest
where the plasma distribution function quickly varies, which is expected, since this is where the numerical
differentiation from computing (\ref{eq:guass_law_residual}) has the largest truncation error.
The error stems from solving Poisson's equation with SBP FD using the initial condition in (\ref{eq:diocotron_ic}).
Poisson's equation is solved so that the initial electromagnetic fields are consistent with the initial distribution of plasma.
Due to the stabilization, the initial residual decays slightly in time.
\subsection{Weibel-type instability}
We now simulate a 2D2V Weibel-type instability \cite{Liu_et_al_2025}, to further verify
the solver and its efficiency. In contrast to the diocotron instability, we do not add
an external magnetic field. We therefore solve (\ref{eq:vlasov_Bext}) and (\ref{eq:maxwell_Bext})
with $\bB_{\text{ext}}=0$ and $\varepsilon=1$.
The initial ion distribution is set as
\begin{equation}\label{eq:weibel_ic}
  f(x, y, v, w) = \frac{1}{4\pi v_t^2} \exp \left( -\frac{v^2}{2v_t^2} \right) \left(\exp \left( -\frac{(w - u_d)^2}{2v_t^2} \right) + \exp \left( -\frac{(w + u_d)^2}{2v_t^2} \right) \right).
\end{equation}Instead of solving Poisson's equation, the electromagnetic fields are initialized as a bath of fluctuations,
\begin{equation}\label{eq:weibel_fields}
  \begin{aligned}
    E_x(x, y) &= \sum_{i=-N_m}^{N_m} \sum_{j=-N_m}^{N_m} \tilde{E}_{x,i,j} \sin \left( i k_x x + j k_y y + \tilde{\phi}^{E_x}_{i,j} \right),\\
    B_z(x, y) &= \sum_{i=-N_m}^{N_m} \sum_{j=-N_m}^{N_m} \tilde{B}_{z,i,j} \sin \left( i k_x x + j k_y y + \tilde{\phi}^{B_z}_{i,j} \right),\\
    E_y(x, y) &= \sum_{i=-N_m}^{N_m} \sum_{j=-N_m}^{N_m} \tilde{E}_{y,i,j} \sin \left( i k_x x + j k_y y + \tilde{\phi}^{E_y}_{i,j} \right),
  \end{aligned}
\end{equation}where $\tilde{E}_{x,i,j}$, $\tilde{B}_{z,i,j}$ and $\tilde{E}_{y,i,j}$ are random amplitudes
with average amplitudes around $10^{-8}$, $\tilde{\phi}_{i,j}$ are random phases and $N_m=8$.
As done in \cite{Liu_et_al_2025}, we set the wave numbers to 
$k_x = k_{\max}^{\text{FI}}$ and $k_y = k_{\max}^{\text{TS}}/3$, where $k_{\max}^{\text{FI}} = 2.31$ 
corresponds to the fastest-growing filamentation mode and $k_{\max}^{\text{TS}} = 6.14$ corresponds 
to the fastest-growing two-stream mode.
Drift and thermal velocity are set as $u_d = 0.1$ and $v_t = 0.1 u_d$, respectively.
We solve in phase space $[0, L_x] \times [0, L_y] \times [-v_m, v_m] \times [-w_m, w_m]$
$= [0, 2\pi/k_x] \times [0, 2\pi/k_y] \times [-0.3, 0.3] \times [-0.4, 0.4]$.
Phase space is discretized with $m_x m_y m_v m_w = 128^4$ points.
We solve using $\text{CFL}=0.8$ and
simulation time $T=50$. The time-step becomes $\Delta t=0.011$, $0.17$ seconds per time-step and the simulation takes about 12 minutes to run in total.
The resulting distribution function at $t=50$ is shown in Figure \ref{fig:wiebel_instability_2D},
projected onto the $(x,v)$ and $(y,w)$ planes, and in Figure \ref{fig:wiebel_instability_3D} as a
three-dimensional view of $f(x,y,w)$.

\begin{figure}[h] %
  \centering
  \begin{subfigure}[b]{0.4\linewidth}
    \centering
    \includegraphics[width=\linewidth]{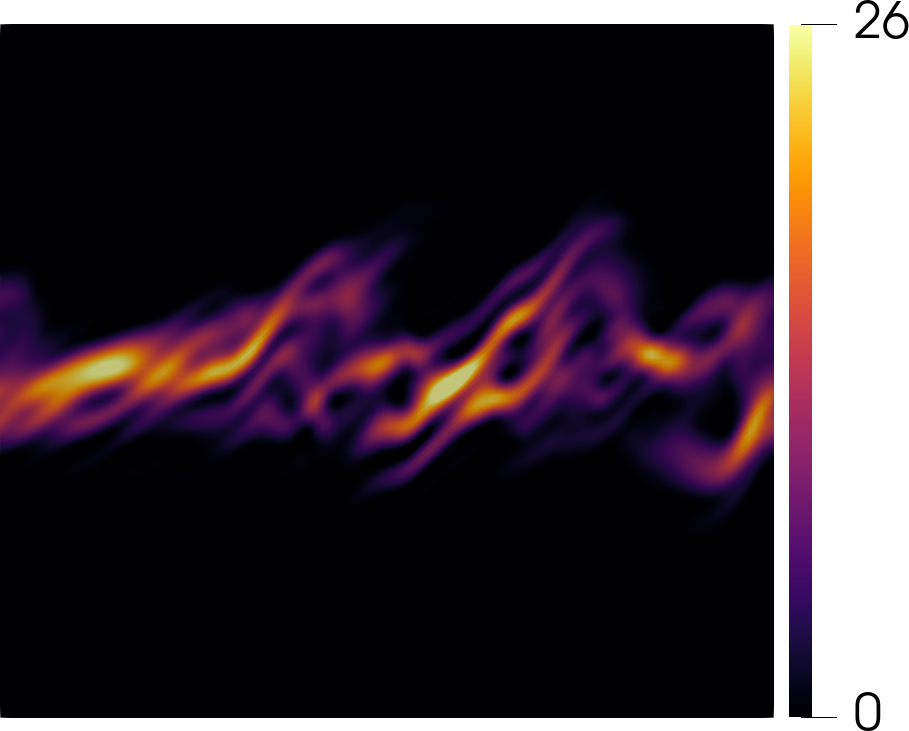}
    \caption{$f(x,v)$ at $t=50$}
  \end{subfigure}
  \hspace{0.2in}
  \begin{subfigure}[b]{0.4\linewidth}
    \centering
    \includegraphics[width=\linewidth]{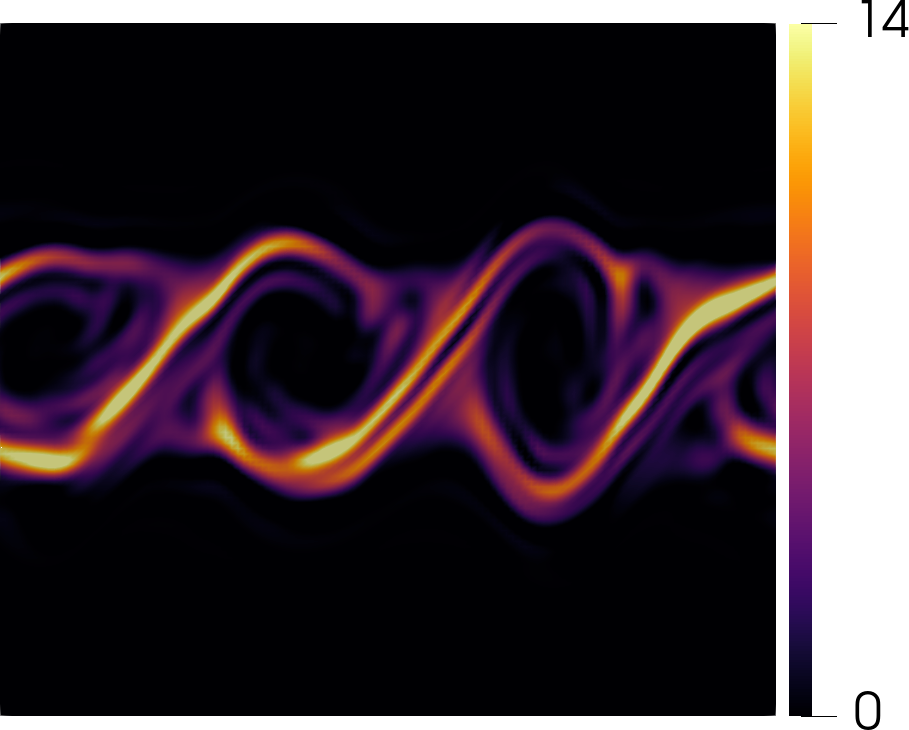}
    \caption{$f(y,w)$ at $t=50$}
  \end{subfigure}
  
  \caption{Weibel instability: Plotted in two dimensions. The horizontal and vertical axes represent position and velocity, respectively.}
  \label{fig:wiebel_instability_2D}
\end{figure}

\begin{figure}[h]
  \centering
  \includegraphics[width=0.5\linewidth]{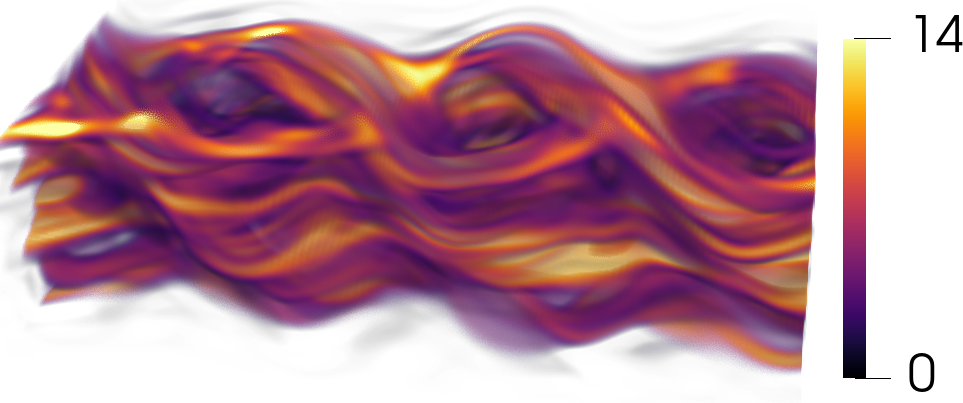}
  \caption{Weibel instability: $f(x, y, w)$ at $t=50$. Plotted in three dimensions.}
  \label{fig:wiebel_instability_3D}
\end{figure}
Slight differences in the solutions presented in \cite{Liu_et_al_2025} are due to ions
being simulated instead of electrons, not disclosing the specific value of $N_m$ used and
the stochastic nature of the electromagnetic initialization.

\subsection{Kelvin--Helmholtz instability}
\label{subsec:kh}

Lastly, we present results for a simulation of the Kelvin--Helmholtz instability, by mimicking the
setup in \cite{Umeda_et_al_2010}. This simulation setup allows us to evaluate the solver's ability to 
accurately capture the $\bE \times \bB$ drift and cross-scale kinetic coupling.
Two particle species are considered, namely ions and electrons. We will use $s$ to denote
the species type, where we denote $s=e$ for electrons and $s=i$ for ions. These have
charge $q_s$ and mass $m_s$. We naturally have that $q_i=1$ and $q_e=-1$, and for
computational efficiency, we let $m_i / m_e = 16$ (when in reality, this ratio is much bigger).
We let $m_i=1$, and set the speed of light in Maxwell's equations to $c=80$, which stems from
$c/v_{ti}=80$, where we set the ion thermal velocity $v_{ti}$ to unity.
This gives us the following systems of equations to be solved:
\begin{equation}\label{eq:kh_system}
  \begin{aligned}
    &\partial_t f_s + \bv  \cdot \GRAD_x f_s + \frac{q_s}{m_s}\left( \bE  + \bv  \times \bB  \right) \cdot \GRAD_v f_s = 0,\quad s=i,e,\\
    &\GRAD_{\bx } \times \bE  = -\partial_t \bB , \\
    &\GRAD_{\bx } \times \bB  = \frac{1}{c^2} \left( \partial_t \bE  + \bJ \right) .
  \end{aligned}
\end{equation}
The initial distribution is an MHD equilibrium characterized by a velocity shear layer. The ion density profile
transitions between a high-density cold region and a low-density hot region,
\begin{equation}\label{eq:kh_ni}
  n_i(y) = \frac{n_0}{2} \left[ (1+\gamma)+(1-\gamma)\tanh{(y/L_s)} \right],
\end{equation}
where $n_0$ is the reference density at $y=-\infty$ and $\gamma=0.1$ is the asymptotic density ratio.
The electron number density is slightly modified to satisfy Gauss' law,
\begin{equation}\label{eq:kh_ne}
  n_e(y) = n_i(y) + \frac{\epsilon_0 B_0 u_0}{2L_sq_i} \frac{1}{\cosh^2{(y/L_s)}}.
\end{equation}
Hence, the system is not charge neutral.
The macroscopic drift velocity in the x-direction is given by
\begin{equation}\label{eq:kh_drift}
  u_{x,s}(y) = -\frac{u_0}{2} \tanh{(y/L_s)}.
\end{equation}

To maintain a uniform plasma thermal pressure across the simulation domain, the local thermal velocity must scale with the density profile. 
Thus, we define $v_{ts}(y) = v_{ts,0} \sqrt{n_0 / n_s(y)}$, where $v_{ts,0}$ is the base thermal velocity at $y=-\infty$.
The particles are loaded by the shifted Maxwell distribution with drift velocity $u_0(y)$ and thermal velocity $v_{ts}(y)$,
\begin{equation}\label{eq:kh_ic}
  f_s(x, y, v, w) = \frac{n_s(y)}{2\pi v_{ts}^2(y)} \exp \left( - \frac{(v - u_{x,s}(y))^2 + (w - \delta w(x))^2}{2 v_{ts}^2(y)} \right).
\end{equation}
We solve in phase space $(x, y) \in [0, 11.2L_s] \times [-7L_s, 7L_s]$.
Parameters are chosen such that the most unstable Kelvin--Helmholtz vortex is located at
the maximum wave number $k_{\text{max}}=2\pi / L_x \approx 0.56$. Note that we include a perturbation
$\delta w(x)=\delta u_0 \sin(k_x x)$, where $\delta u_0=0.01u_0$, to jumpstart the instability.
Periodic BCs are imposed for the $x$-direction, while open BCs are used for the $y$-direction.
This is done by clamping the indices in the $y$-direction to the endpoints, 
effectively applying a homogeneous Neumann condition at the edges.
Lastly, electromagnetic fields are initialized as
\begin{equation}\label{eq:kh_Ey}
  E_y(y) = -\frac{B_0 u_0}{2} \tanh{(y/L_s)}
\end{equation}
where $B_0 = \omega_{ci} m_i / q_i$.
The full set of parameters used in this setup is summarized in Table \ref{tab:kh_parameters}.
\begin{table}[h]
  \centering
  \small
  \begin{tabular}{l c c}
    \hline
    Parameter & Symbol & Value \\
    \hline
    Ion-to-electron mass ratio & $m_i/m_e$ & 16 \\
    Speed of light & $c/v_{ti}$ & 80.0 \\
    Half thickness of shear layer & $L_s$ & 8.0 $r_i$ \\
    Alfv\'{e}n velocity & $V_A$ & 7.0 $v_{ti}$ \\
    Ion sound velocity & $V_S$ & 2.0 $v_{ti}$ \\
    Velocity shear & $u_0$ & 7.28 $v_{ti}$ \\
    Ion cyclotron to plasma frequency ratio & $\omega_{ci}/\omega_{pi}$ & 0.0875 \\
    Ion-to-electron temperature ratio & $T_i/T_e$ & 1 \\
    Asymptotic number density ratio & $\gamma$ & 0.1 \\
    Configuration space grid cells & $m_x \times m_y$ & 128 $\times$ 160 \\
    \hline
  \end{tabular}
  \caption{Simulation quantities and parameters for the Kelvin--Helmholtz instability setup from \cite{Umeda_et_al_2010}.}
  \label{tab:kh_parameters}
\end{table}
We solve using $\text{CFL}=0.5$ and %
simulation time $T=120$. The time-step becomes $\Delta t=0.00057$, $0.18$ seconds per time-step and the simulation takes about 10.8 hours to run in total.

\begin{figure}[h]
  \centering
  \begin{subfigure}[b]{0.4\linewidth}
    \centering
    \includegraphics[width=\linewidth]{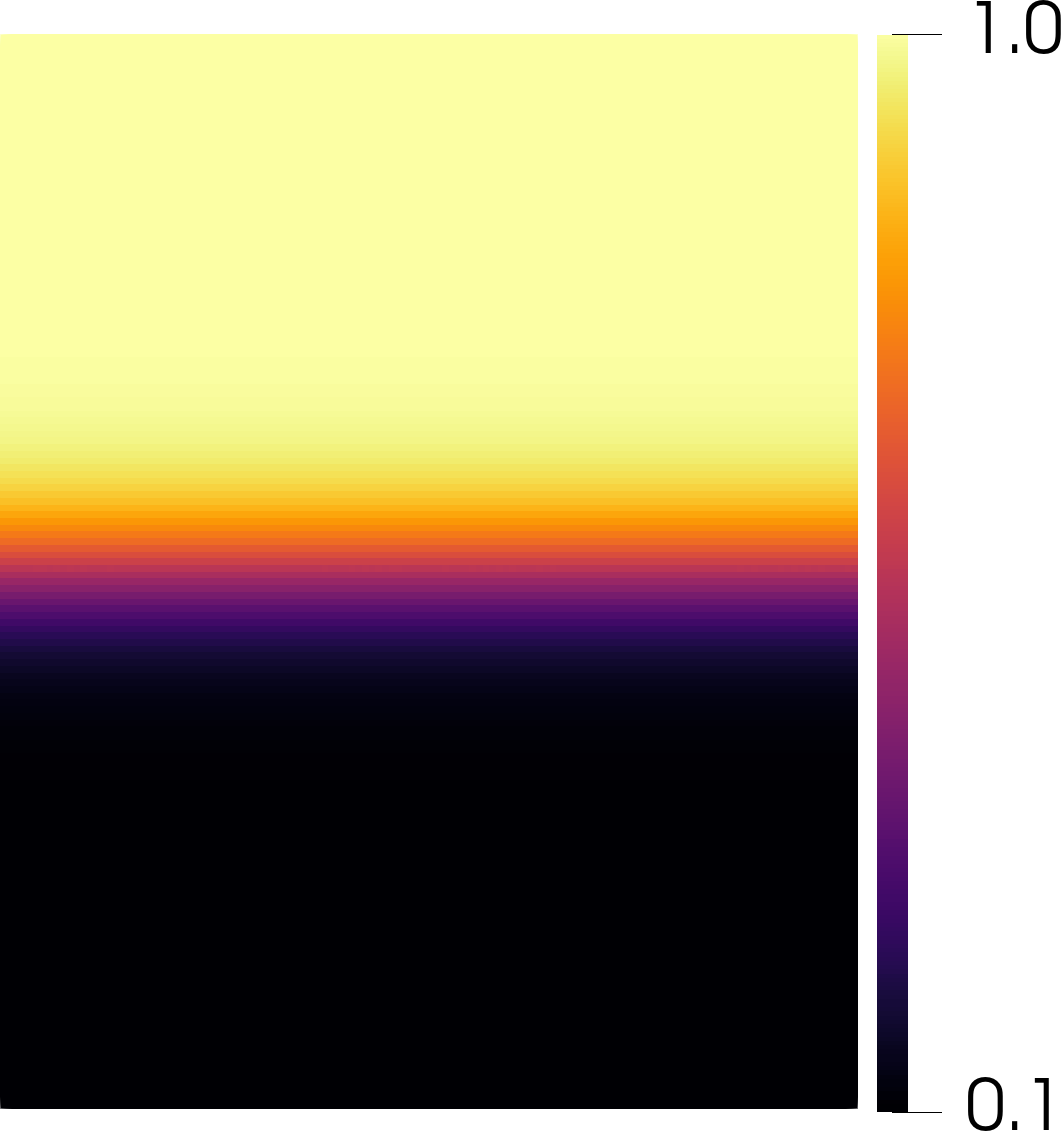}
    \caption{$t=0$}
  \end{subfigure}
    \hspace{0.2in}
    \begin{subfigure}[b]{0.4\linewidth}
        \centering
        \includegraphics[width=\linewidth]{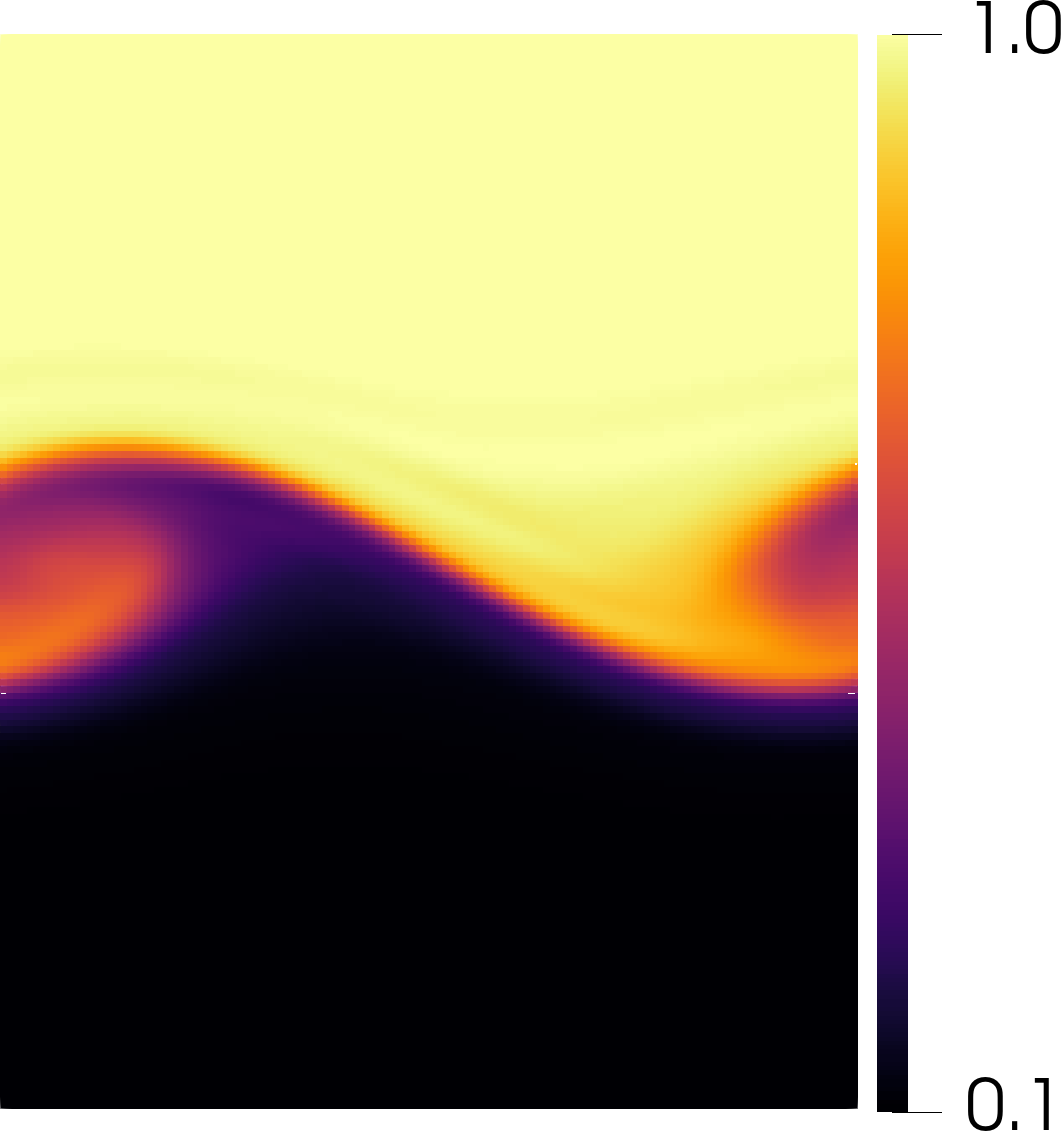}
        \caption{$t=60$}
    \end{subfigure}
    \\
    \begin{subfigure}[b]{0.4\linewidth}
        \centering
        \includegraphics[width=\linewidth]{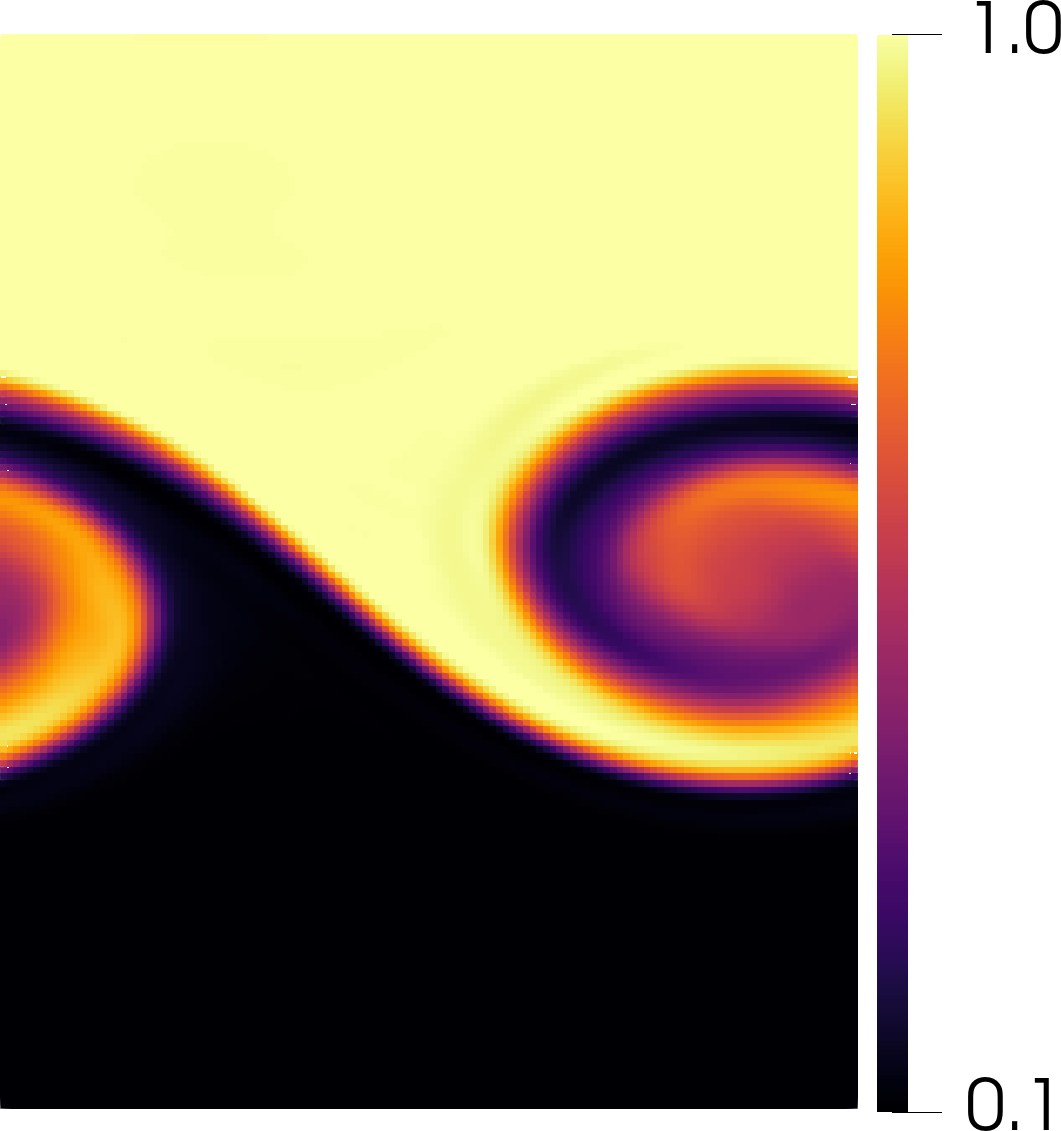}
        \caption{$t=80$}
    \end{subfigure}
    \hspace{0.2in}
    \begin{subfigure}[b]{0.4\linewidth}
        \centering
        \includegraphics[width=\linewidth]{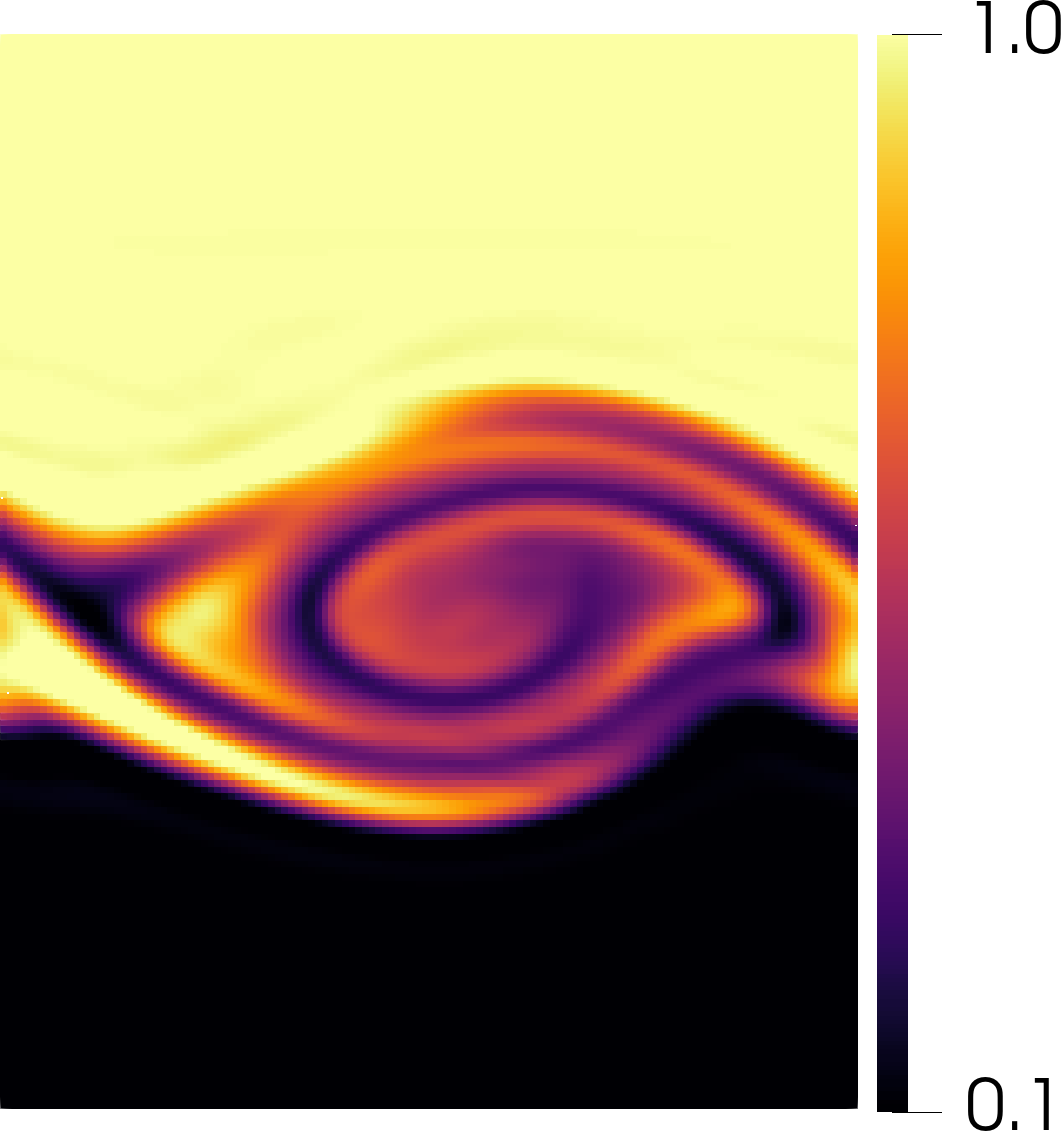}
        \caption{$t=120$}
    \end{subfigure}
    
    \caption{Kelvin--Helmholtz instability: $m_x=128$, $m_y=160$ and $m_v=m_w=80$ solution of $f(x, y, t)$. The horizontal and vertical axes represent position $x$ and $y$, respectively.}
    \label{fig:kelvin_helmholtz}
\end{figure}
As in \cite{Umeda_et_al_2010}, the instability grows from $k_{\text{max}}$. 
However, our setup jumpstarts the instability with an explicit perturbation at $k_x = k_{\text{max}}$. 
By contrast, the reference simulation is initiated with an unperturbed MHD equilibrium. 
Because this MHD equilibrium takes several ion cyclotron periods to approach a true Vlasov--Maxwell 
equilibrium, during which the spatial profile of the ion density is slightly modified, the reference 
instability grows at a slightly different rate than ours. Despite this difference, the structural 
development of the instability shown in Figure \ref{fig:kelvin_helmholtz} remains similar.

\section{Conclusions}
\label{Sec:conclusions}
In this paper, we have developed a high-order upwind SBP finite difference discretization of the 2D2V
Vlasov--Maxwell system, built from tensor products of one-dimensional operators.
The scheme is stabilized by Lax--Friedrichs flux splitting of the advection terms and integrated 
in time using the fourth-order, five-stage strong stability
preserving explicit Runge--Kutta method.
The implementation is matrix-free in CUDA, where no differentiation operator is ever assembled
and every operator application reduces to a local stencil evaluation, yielding low memory footprint.
The semi-discrete scheme conserves mass and momentum,
and satisfies $\p_t\|\vf\|_{\sfH_\Omega}^2 \le 0$.
Fully discrete, mass is conserved exactly and momentum up to truncation error. Total energy is not conserved
due to upwind stabilization that dissipates it.
Numerically we verify the high order of accuracy and reproduce the diocotron, Weibel
and Kelvin--Helmholtz instabilities, the largest run using 420 million degrees of freedom on a single NVIDIA L40.
Extending the framework to 3D3V is the natural next step, and will require multiple GPUs and likely an implicit
treatment of the stiff magnetized regime.

\section*{Funding}
This research is funded by Swedish Research Council (VR) under grant
number 2025-04764 and 2021-04620.

\section*{Acknowledgements}
The computations were enabled by resources in project UPPMAX 2026/1-45 provided
by Uppsala University at UPPMAX.

\newpage
\bibliographystyle{abbrvnat} 
\bibliography{ref}
\end{document}